\documentclass{amsart}

\usepackage[dvips]{epsfig}
\usepackage{graphicx}
\usepackage{latexsym}
\usepackage{amsmath}
\usepackage{amsthm}
\usepackage{amssymb}
\usepackage{ulem}

\usepackage{amsmath,amsthm,amsfonts,amssymb,amscd,amsbsy,dsfont,hyperref}

\usepackage{palatino}

\usepackage{eucal}
\usepackage{float}
\usepackage{xcolor}
\usepackage{multirow} 

\newtheorem{thm}{Theorem}
\newtheorem{assumption}[thm]{Assumption}
\newtheorem{lem}[thm]{Lemma}

\newtheorem{cor}[thm]{Corollary}
\newtheorem{defi}[thm]{Definition}

\newtheorem{prop}[thm]{Proposition}
\newtheorem{rk}[thm]{Remark}

\newenvironment{preuve}{\vip \noindent {\it Proof}}{\hfill$\square$\vip}

\newcommand{\field}[1]{\mathbb{#1}}
\newcommand{\EE}{\field{E}}
\newcommand{\E}{\field{E}}

\newcommand{\PP}{\field{P}}

\newcommand{\RR}{\field{R}}
\newcommand{\R}{\field{R}}

\newcommand{\vip}{\vskip.2cm}

\allowdisplaybreaks

\begin{document}

\date{\today}

\title[Parametric estimation for interacting particle systems]{The LAN property for McKean-Vlasov models in a mean-field regime}

\author{Laetitia Della Maestra and Marc Hoffmann}

\address{Laetitia Della Maestra, Universit\'e Paris-Dauphine \& PSL, CNRS, CEREMADE, 75016 Paris, France}
\email{dellamaestra@ceremade.dauphine.fr}
\address{Marc Hoffmann, Universit\'e Paris-Dauphine \& PSL, CNRS, CEREMADE, 75016 Paris, France}
\email{hoffmann@ceremade.dauphine.fr}

\begin{abstract} 
We establish the local asymptotic normality (LAN) property for estimating a multidimensional parameter in the drift of a system of $N$ interacting particles observed over a fixed time horizon in a mean-field regime $N \rightarrow \infty$. By implementing the classical theory of Ibragimov and Hasminski, we obtain in particular sharp results for the maximum likelihood estimator that go beyond its simple asymptotic normality thanks to H\'ajek's convolution theorem and strong controls of the likelihood process that yield asymptotic minimax optimality (up to constants). Our structural results shed some light to the accompanying nonlinear McKean-Vlasov experiment, and enable us to derive simple and explicit criteria to obtain identifiability and non-degeneracy of the Fisher information matrix. These conditions are also of interest for other recent studies on the topic of parametric inference for interacting diffusions. 
\end{abstract}

\maketitle



\tableofcontents

\section{Introduction}
 

\subsection{Motivation}

Collective dynamics models are becoming increasingly popular in modelling complex stochastic systems, with a versatiliy of applications, ranging from mathematical biology (neurosciences, Baladron {\it et al.} \cite{Baladron}, structured models in population dynamics, Mogilner {\it et al.} \cite{mogilner99}, Burger {\it et al.} \cite{BURGER}) to social sciences (opinion dynamics, Chazelle {\it et al.} \cite{CHAZELLE}, cooperative behaviours, Canuto {\it et al.} \cite{canuto12}) and finance (systemic risk, Fouque and Sun \cite{fouque2013systemic}), or more recently, mean-field games (Cardaliguet {\it et al.} \cite{CARDA}, Cardaliaguet and Lehalle \cite{LC19}). Whereas stochastic systems of interacting particles and associated nonlinear Markov processes in the sense of McKean \cite{MCKEAN} date back to the 1960's and have been studied extensively over more than half a century, see {\it e.g.} \cite{bolley2007quantitative, Sznitman1, Sznitman2, meleard1996asymptotic,  TANAKAHITSUDA} among a myriad of references, the development of statistical inference in this setting is only emerging, (with some notable exceptions like L\"ocherbach \cite{locherbach2002lan} in large time or {\it e.g.} Kasonga \cite{kasonga1990maximum} or Bishwal \cite{bishwal2011estimation}) in a mean-field limit. Recently, Giesecke {\it et al.}  \cite{giesecke2020inference} and Sharrock, Kantas, Parpas and Grigorios \cite{sharrock2021parameter} revisit the work of Kasonga and consider a parametric framework where  convergent and asymptotically normal contrast estimators are constructed.  Several other parametric frameworks (that consider various observation schemes and asymptotic frameworks) have also been recently considered, like \cite{chen2020maximum, liu2020parameter, wen2016maximum} or Genon-Catalot and Laredo \cite{genon2021parametric, genon2021probabilistic}. There also exist recent results in nonparametric inference: we mention our work \cite{della2022nonparametric} and Belometsny {\it et al.} \cite{belomestny2021semiparametric}, together with studies in identification like \cite{lang2021identifiability,  li2020coercivity, li2020identifiability} or  learning \cite{ lang2020learning, lu2020learning, lu2021learning}.\\

The present paper, close in spirit to \cite{bishwal2011estimation, giesecke2020inference, kasonga1990maximum} and \cite{sharrock2021parameter} (in their so-called offline case) considers a parametric framework in a mean-field regime over a fixed time horizon.  We take a deeper look at the asymptotic structure of the associated statistical experiment, in the sense of local asymptotic normality or LAN, in order to derive strong results for the maximum likelihood, both in asymptotic distribution and in an asymptotic minimax sense (up to constants) for various loss functions. For simplicity, we keep-up with continuous observations, but we briefly explain how to move to a discrete data setting. Also, we look for simple and explicit criteria that enable us to verify identifiability and non-degeneracy of the model. This is a non-trivial issue in the context of nonlinear McKean-Vlasov models that is usually a bit overlooked in the literature.

\subsection{Setting}
 We have 
a parameter of interest $\vartheta $ lying in a compact set $\Theta \subset \R^p$ (with non empty interior), for some fixed $p \geq 1$. For some fixed time horizon $T>0$, we continuously observe a stochastic system of $N$ interacting particles 
\begin{equation} \label{eq: data}
X^{(N)} = (X_t^1,\ldots, X_t^{N})_{t \in [0,T]},
\end{equation}
evolving in an Euclidean ambient space $\R^d$, that solves
\begin{equation} \label{eq:diff basique}
\left\{
\begin{array}{l}
dX_t^{i} = b(\vartheta ; t, X_t^{i},\mu^{(N)}_t)dt+\sigma(t,X_t^{i}) dB_t^i,\;\;1 \leq i \leq N,\; t \in [0,T],\\  
\mathcal L(X_0^{1},\ldots, X_0^{N}) = \mu_0^{\otimes N},
\end{array}
\right.
\end{equation}
where 
$\mu^{(N)}_t= N^{-1} \sum_{i = 1}^N \delta_{X_t^{i}}$ is the empirical measure of the system.
The $(B_t^i)_{t \in [0,T]}$
are independent $\R^d$-valued Brownian motions. The initial condition $\mu_0$, the drift $b$ and the diffusion coefficient 
$\sigma$ are at least sufficiently regular 
so that 
$$\mu^{(N)} = (\mu^{(N)}_t)_{t \in [0,T]} \rightarrow \mu = (\mu_t^{})_{t \in [0,T]}$$ weakly as $N \rightarrow \infty$,
where $\mu$ is a family of probability measures that solves (in a weak sense) the parabolic nonlinear equation
\begin{equation}  \label{eq: mckv approx}
\left\{ 
\begin{array}{ll}
\partial_{t}\mu + \mathrm{div}\big( b(\vartheta ; \cdot,\mu) \mu\big) = \tfrac{1}{2}\sum_{k,k'=1}^d\partial_{kk'}^2\big(c_{kk'}\mu\big),\; t \in [0,T],\\ 
 \mu_{t=0}^{} = \mu_0,
\end{array}
\right. 
\end{equation} 
with $c = \sigma \sigma^{\top}$.
We will write $\mu^{\vartheta} = (\mu^{\vartheta}_t)_{t \in [0,T]} $ to emphasise the dependence in $\vartheta$. In this context, we are interested in estimating from data \eqref{eq: data} the parameter $\vartheta \in \Theta $ of the function $(\vartheta ;  t, x,\nu) \mapsto b(\vartheta; t, x,\nu) \in \R^d$. Asymptotics are taken as  $N \rightarrow \infty$.\\

A particular case of interest that covers many examples is when the dependence in the measure variable for $b$ is linear: we then have
\begin{equation} \label{eq: explain drift}
b(\vartheta ; t,X_t^i,\mu^{(N)}_t) = \int_{\R^d}\widetilde b(\vartheta; X_t^i,y)\mu_t^{(N)}(dy) = N^{-1}\sum_{j = 1}^N \widetilde b(\vartheta; X_t^i,X_t^j),
\end{equation}
for some function $\widetilde b: \Theta \times \R^{d} \times \R^d \rightarrow \R^d$. A typical form is $\widetilde b(\vartheta; t,x,y) = G_\vartheta(x)+F_\vartheta(x-y)$ where $G_{\vartheta}, F_{\vartheta}:\R^d \rightarrow \R^d$ play the role of a common external force to the system and an interaction force respectively.

\subsection{Results and organisation of the paper}
In Section \ref{section : organisation}, we rigorously construct the (sequence of) statistical experiment(s)  generated by the observation \eqref{eq: data} under the dynamics \eqref{eq:diff basique} that we denote $(\mathcal E^N)_{N \geq 1}$. It is well defined and regular in the classical sense of Ibragimov and Hasminski \cite{ibragimov2013statistical} under strong integrability of the initial condition $\mu_0$ and standard smoothness assumptions on the drift $\vartheta \mapsto b(\vartheta;\cdot)$ and the diffusion matrix $c = \sigma \sigma^\top$, see Assumptions \ref{ass: init condition}, \ref{reg sigma}, \ref{reg b} and \ref{hyp : reg sigma b} and Proposition \ref{prop: regularity model}. The deep study of the identifiability of $\mathcal E^N$ and the non-degeneracy of its information matrix $\mathbb I_{\mathcal E^N}(\vartheta)$ is simplified via the accompanying experiment $\mathcal G^{\otimes N}$, where $\mathcal G$ is generated by the continuous observation of a solution to the McKean-Vlasov equation
$$\left\{
\begin{array}{l}
dX_t = b(\vartheta;t,X_t,\mu_t^\vartheta) dt + \sigma(t,X_t)dB_t,\; t \in [0,T],\\  
\mathcal L(X_0) = \mu_0,
\end{array}
\right.
$$
for a standard Brownian motion $(B_t)_{t \in [0,T]}$ on $\R^d$ and where $\mu_t^\vartheta$ is the marginal distribution of the solution at time $t$. In particular, in the case of representation \ref{eq: explain drift} we have that $\mathcal E^N$ and $\mathcal G^{\otimes N}$ do not separate asymptotically  by a simple entropy argument, see Proposition \ref{prop: indisting}, and we always have the convergence of the corresponding Fisher information matrices:
$$N^{-1}\mathbb I_{\mathcal E^N}(\vartheta) \rightarrow \mathbb I_{\mathcal G}(\vartheta)$$
in a mean-field limit $N \rightarrow \infty$, as established in  Proposition \ref{prop: fisher cont}. This approximation is the gateway to obtain explicit identifiability and non-degeneracy criteria, as detailed in Section \ref{sec: identif and non-deg}. In particular, under additional regularity assumptions, we obtain a quite simple criterion for $\mathbb I_\mathcal G(\vartheta)$ to be non-degenerate in Proposition \ref{prop: criterion fisher}, namely the property that one of the functions
\begin{equation} \label{eq: crit informal}
x \mapsto  \nabla_{\vartheta}(c^{-1/2}b)^j(\vartheta;0,x,\mu_0)^{\top} z,\;\;\;j = 1,\ldots, d
\end{equation}
is not identically vanishing, for every $z \in \R^p$ with $|z|=1$, with $c^{-1/2}$ a square root of $c = \sigma \sigma^\top$. We use the notation $f = (f^j)_{1 \leq j \leq d}$ componentwise, the $f^j$ being real-valued functions. In particular, \eqref{eq: crit informal} has the advantage to only relate to the initial condition $\mu_0$ in the measure argument and not the whole $(\mu_t^\vartheta)_{t \in [0,T]}$ which is (almost) never explicit. Having a simple criterion to achieve the non-degeneracy of the Fisher information seems to have been a bit overlooked in the literature (where it is usually simply assumed to hold true) and our result is thus of interest for other studies.\\

In Section \ref{sec: main}, we state the main results of the paper, Theorem \ref{thm: LAN prop}, where we establish the LAN property: if we reparametrise the experiments via $\vartheta = \vartheta_0 + N^{-1/2}u$ locally around a fixed point $\vartheta_0$, with $u \in \R^p$ being now the unknown parameter, then both $\mathcal E^N$ and $\mathcal G^{\otimes N}$ look like a Gaussian shift: we observe 
$$Y^N = u + \mathbb I_{\mathcal G}(\vartheta_0)^{-1/2}\xi,$$
where $\xi$ is a standard Gaussian random vector in $\R^p$. This has important consequences in terms of existence and properties of optimal procedures: we have H\'ajek's convolution theorem (Corollary \ref{cor strong LAN}), namely for any estimator $\widehat \vartheta_N$,
\begin{equation} \label{eq: informal hajek}
\liminf_{N \rightarrow \infty}\sup_{|\vartheta'-\vartheta| \leq \delta} \E_{\PP_{\vartheta'}^N}\big[w\big(N^{1/2}\mathbb I_{\mathcal G}(\vartheta)^{1/2}(\widehat \vartheta_N-\vartheta')\big)\big] \geq  \E[w(\xi)],
\end{equation}
for small enough $\delta>0$, where $\PP_{\vartheta'}^N$ is the distribution of the data when the parameter is $\vartheta'$ and  $w$ is an arbitrary loss function satisfying some regularity properties. The bound \eqref{eq: informal hajek} is achieved by the maximum likelihood estimator $\widehat \vartheta_N^{\,\tt{mle}}$ obtained by maximising the contrast
\begin{equation} \label{eq: max constrast}
\vartheta \mapsto \ell^N(\vartheta; X^{(N)}) = \sum_{i=1}^N  \int_0^T  \Big( (c^{-1}b)(\vartheta; t,X^{i}_t,\mu^{(N)}_t)^{\top} dX^i_t
-\tfrac{1}{2} |(c^{-1/2}b)(\vartheta; t,X^{i}_t,\mu^{(N)}_t)|^2 dt\Big).
\end{equation}
This implies in particular the convergence
\begin{equation} \label{eq: asymp classical}
\sqrt{N}\big(\widehat \vartheta_N^{\,\tt{mle}}-\vartheta\big) \rightarrow \mathcal N\big(0, \mathbb I_{\mathcal G}(\vartheta)^{-1}\big)
\end{equation}
in distribution.
Moreover, we have in Theorem \ref{thm: MLE prop} the minimax asymptotic optimality of $\widehat \vartheta_N^{\,\tt{mle}}$, in the sense that
$$\mathcal R_w^N(\widehat \vartheta_N^{\,\tt{mle}}; \Theta)= \inf_{\widehat \vartheta_N}\mathcal R_w^N(\widehat \vartheta_N; \Theta)(1+o(1))$$
where $\mathcal R_w^N(\widehat \vartheta_N; \Theta) = \sup_{\vartheta \in \Theta}\E_{\PP_\vartheta^N}[w(N^{1/2}\mathbb I_{\mathcal G}(\vartheta)^{1/2}(\widehat \vartheta_N-\vartheta))]$ is the classical minimax risk. Thus the LAN property enables us to obtain considerably stronger results than simply \eqref{eq: asymp classical}. In Section \ref{sec: examples}, we investigate several non-trivial examples that generalise the results of \cite{kasonga1990maximum}, and where our identifiability and non-degeneracy criteria easily apply. We treat in particular the case of a kinetic mean-field double layer potential that may serve as a representative model  for swarming models, see in particular \cite{bolley2011stochastic} and the references therein. The proofs are delayed until Sections \ref{section : proofs} and \ref{sec: remaining proofs}, with an appendix (Section \ref{sec: appendix}) that contains useful technical results.\\

In practice, maximising the function \eqref{eq: max constrast} is not feasible, since only discrete data are available. It is then reasonable to replace the ideal observation \eqref{eq: data} by the more realistic
$$
X^{(N, m)} = \big(X_{t}^1,\ldots, X_{t}^{N}\big)_{t \in \{t_0^m, \ldots, t_m^m\}},
$$
where $(0 = t_0^m < t_1^m < \ldots < t_m^m =T)$ is a subdivision of $[0,T]$ with mesh 
$$\max_{1 \leq j \leq m}(t_j^m-t_{j-1}^m) \leq Cm^{-1}.$$
We thus have $(m+1) \times N$ data with values in $\R^d$. We may then replace \eqref{eq: max constrast} by
\begin{align*}
\vartheta & \mapsto N^{-1}\sum_{i=1}^N\sum_{j =0}^m\Big((c^{-1}b)(\vartheta; t_j^m,X^{i}_{t_{j-1}^m},\mu^{(N)}_{t_{j-1}^m})^{\top} (X^i_{t_j^m}-X^i_{t_{j-1}^m}) \\
&-\tfrac{1}{2}|(c^{-1/2}b)(\vartheta; t_{j-1}^m,X^{i}_{t_{j-1}^m},\mu^{(N)}_{t_{j-1}^m})|^2 (t_j^m-t_{j-1}^m)\Big).
\end{align*}
Assuming the function $(t, x) \mapsto (c^{-1/2}b)(\vartheta; t,x,\mu_{t}^{(N)})$ to be smooth, we may safely expect the discrete approximation to be close to its continuous counterpart up to an additional error of order $m^{-1/2}$, by standard high-frequency discretisation techniques, see the textbooks of Jacod and co-authors \cite{ait2014high, jacod2011discretization, JS}. In particular, if $m \gg N$, the same results as for continuous observations are likely to hold true.\\

\section{Construction and properties of the statistical model}\label{section : organisation}
\subsection{Notation}
The dimension $d \geq 1$ of the state space $\R^d$ and the dimension $p \geq 1$ of the parameter space $\Theta$ as well as the time horizon $T>0$ are fixed once for all. 
We write $|\cdot |$ for the Euclidean distance on $\R^q$  ($q = p,d$ or any other integer, depending on the context) or for a matrix norm on $\R^p \otimes \R^p$ fixed 
throughout.\\


We consider functions that are mappings defined on products of metric spaces (typically $\Theta \times [0,T] \times \R^d \times \mathcal P_1$ or subsets of these) with values in $\R$ or $\R^d$. Here, $\mathcal P_1$ denotes the set of probability measures on $\mathbb{R}^{d}$ with a first moment, endowed with the Wasserstein $1$-metric
$$\mathcal W_1(\mu,\nu) = \inf_{m \in \Gamma(\mu,\nu)} \int_{\R^d \times \R^d} \big|x-y\big|m(dx,dy) = \sup_{|\phi|_{\mathrm{Lip}}\leq 1}\int_{\R^d}\phi\, d\big(\mu-\nu\big),$$
where $\Gamma(\mu, \nu)$ denotes the set of probability measures on the product space $\R^d \times \R^d$ with marginals $\mu$ and $\nu$. For a probability measure $\mu$ on $\R^d$, we also set
$$\mathfrak m_r(\mu) = \int_{\R^d} |y|^r \mu(dy)$$
for its moment of order $r \geq 1$ and we say that $\mu \in \mathcal{P}_r$ if $\mathfrak m_r(\mu)$ is finite.
All the functions in the paper are implicitly measurable with respect to the Borel-sigma field induced by the product topology.
A $\R^d$-valued function $f$ is written componentwise as $f = (f^k)_{1 \leq k \leq d}$ where the $f^k$ are real-valued. 
We denote by $\partial_{\vartheta_k}$, $\nabla_{\vartheta}$, $\partial^2_{\vartheta_k \vartheta_l}$ respectively the partial derivative of a function with respect to the $k$-th component $\vartheta_k$, the gradient of a real-valued function with respect to $\vartheta$, the second order partial derivative of a function with respect to the $k$-th and $l$-th components $\vartheta_k , \vartheta_l$.\\ 

Finally, we repeatedly use the notation $C$ for a positive number that does not depend on $N$, nor $\vartheta$, that may vary from line to line and that we call a constant, although it usually depends on some other (fixed) quantities of the model. In most cases, it is explicitly computable. 

\subsection{Model assumptions} \label{section : assumptions}

\subsubsection*{Well-posedness of the model and its associated statistical experiment}

We work under the following strong integrability property for the initial condition $\mu_0$.

\begin{assumption} \label{ass: init condition}
For every $r \geq 1$, we have $\mu_0 \in \mathcal P_r$.
\end{assumption}

 As for the diffusion matrix 
$\sigma: [0,T] \times \R^d \rightarrow \R^d \otimes \R^d$, we make the following strong ellipticity and Lipschitz smoothness assumption. 

\begin{assumption} \label{reg sigma} The diffusion matrix $\sigma$ is measurable and for some $C\geq 0$, we have
$$|\sigma(t,x')-\sigma(t,x)| \leq C|x'-x|.$$
Moreover, $c = \sigma \sigma^{\top}$ is such that $\sigma_-^2 |y|^2 \leq (c(t,x)y)^\top y  \leq \sigma_+^2|y|^2$ for some $\sigma_\pm >0$.
\end{assumption}
As for the drift part $b: \Theta  \times [0,T]\times \R^d \times \mathcal P_1 \rightarrow \R^d$, we work under usual Lipschitz smoothness assumptions. 
\begin{assumption} \label{reg b}
The drift b is measurable and for some $C \geq 0$, we have
$$
\sup_{t \in [0,T], \vartheta \in \Theta}\big|b(\vartheta; t, x',\nu')-b(\vartheta; t,x,\nu)\big| \leq C\big(|x'-x|+\mathcal W_1(\nu',\nu)\big),
$$
and there exists some $\vartheta_0 \in \Theta$ such that
$$b_0 =\sup_{t \in [0,T]} | b(\vartheta_0 ; t, 0, \delta_0)| < \infty \; .$$
\end{assumption}
We let $|b|_{\mathrm{Lip}}$ denote the smallest $C\geq 0$ for which Assumption \ref{reg b}  holds.\\ 

Assumptions \ref{ass: init condition}, \ref{reg sigma}, \ref{reg b} together are sufficient to guarantee the well-posedness of the statistical model: there exists a unique weak solution to \eqref{eq:diff basique} for every $\vartheta \in \Theta$ hence the data $X^{(N)}$ of \eqref{eq: data} is well-defined. More precisely, we let $\mathcal C^N = \mathcal C([0,T],(\R^d)^N)$ denote the space of continuous functions on $(\R^d)^N$, equipped with the natural filtration $(\mathcal F_t)_{t \in [0,T]}$  induced by the canonical mappings  
$$X_t^{(N)}(\omega) = \big(X^1_t(\omega),\ldots ,X^N_t(\omega)\big)= \omega_t.$$ 
For $\mu_0 \in \mathcal P_1$ and $\vartheta \in \Theta$, the probability $\PP^{N}_\vartheta$ on $(\mathcal C^N,\mathcal F^N)$ under which  the canonical process $X^{(N)} = (X_t^{(N)})_{t \in [0,T]}$ is a solution of \eqref{eq:diff basique} for the initial condition $\mu_0^{\otimes N}$ is uniquely defined under Assumptions \ref{ass: init condition}, \ref{reg sigma} and \ref{reg b}. Recommended reference (that covers our set of assumptions) is the textbook by Carmona and Delarue \cite{carmona2018probabilistic} or the lectures notes of Lacker  \cite{lacker2018mean}. Moreover, for every $\vartheta \in \Theta$, the parabolic nonlinear equation \eqref{eq: mckv approx} has a unique probability solution $\mu = (\mu_t^\vartheta)_{t \in [0,T]}$ and  we have the weak convergence $\mu_t^{(N)} \rightarrow \mu_t^\vartheta$ under $\PP_\vartheta^N$, for every $\vartheta \in \Theta$.\\

We thus study under Assumptions \ref{ass: init condition}, \ref{reg sigma}, \ref{reg b} the (sequence of) statistical experiment(s) generated by the observation \eqref{eq: data} under the dynamics \eqref{eq:diff basique} and that we realise as  
$$(\mathcal E^N)_{N \geq 1} = \Big(\mathcal C^N, \mathcal F^N, \big(\PP_\vartheta^N, \vartheta \in \Theta\big)\Big)_{N \geq 1}.$$

Note that at that stage, we do not impose any identifiability assumption {\it i.e.} we do not assume that the mapping $\vartheta \mapsto \PP_\vartheta^N$ is one-to-one. We will discuss that matter together with the non-degeneracy of the model later in Section \ref{sec: identif and non-deg}.
%
%

\subsubsection*{Regularity of the experiment $\mathcal E^N$} 
In order to study the regularity of the model, we need specific smoothness properties for the function $\vartheta \mapsto b(\vartheta, \cdot)$.
\begin{assumption} \label{hyp : reg sigma b}  There exist $r_1,r_2 \geq 1$ and $C>0$ such that for every point $\vartheta $  in the interior of $\Theta$, the function
$\vartheta \mapsto b(\vartheta; t,x,\nu) $ is twice differentiable and for every $1 \leq \ell,\ell' \leq p$,
$$\sup_{t \in [0,T]}(|\partial_{\vartheta_\ell}b(\vartheta;t,x,\nu)| + |\partial_{\vartheta_\ell \vartheta_{\ell'}}^2b(\vartheta;t,x,\nu)|)  \leq C(1 + |x|^{r_1} + \mathfrak m_{r_2}(\nu)),$$
$$\sup_{t \in [0,T]}|\partial_{\vartheta_\ell}b(\vartheta;t,x',\nu')- \partial_{\vartheta_\ell}b(\vartheta;t,x,\nu)| \leq C (|x'-x| + \mathcal W_1(\nu',\nu) ).$$
\end{assumption}
The smoothness properties of the map $\vartheta \mapsto b(\vartheta; \cdot)$ granted by Assumption \ref{hyp : reg sigma b} enables us to explore further the regularity of the experiment $\mathcal E^N$. First, note that we have a log-likelihood by setting
\begin{equation}\label{log-likelihood Girsanov}
\ell^N(\vartheta; X^{(N)}) = \sum_{i=1}^N  \int_0^T  (c^{-1}b)(\vartheta; t,X^{i}_t,\mu^{(N)}_t)^{\top} dX^i_t
-\frac{1}{2} \sum_{i=1}^N \int_0^T |(c^{-1/2}b)(\vartheta; t,X^{i}_t,\mu^{(N)}_t)|^2 dt,
\end{equation}
where $c^{-1/2}$ is fixed once for all. Indeed, by Girsanov's theorem again, the laws $\PP_\vartheta^N$ are all absolutely continuous w.r.t. $\mathbb W^N$, defined as the unique probability on $(\mathcal C^N,\mathcal F^N)$ under which the processes 
$$\Big(\int_0^t c^{-1/2}(s,X_s^i)dX_s^i\Big)_{t \in [0,T]},\;\;1 \leq i \leq N$$ are independent standard Brownian motions on $\R^d$, together with $\mathcal L(X_0^1,\ldots, X_0^N) = \mu_0^{\otimes N}$. In turn, for every $\vartheta \in \Theta$,
$$\frac{d\PP_{\vartheta}^N}{d\mathbb W^N}(X^{(N)}) = \exp\big(\ell^N(\vartheta; X^{(N)})\big)$$
holds 
$\mathbb W^N$-almost-surely.  We further write $\mathcal L^N(\vartheta; X^{(N)}) =  \exp\big(\ell^N(\vartheta; X^{(N)})\big)$ for the likelihood process, indexed by the parameter $\vartheta \in \Theta$. We recall one possible classical definition of a regular statistical experiment, following \cite{ibragimov2013statistical}.
\begin{defi} \label{def: regular}
The dominated (sequence of) experiment(s) $(\mathcal{E}^N)_{N \geq 1}$ is regular if 
\begin{itemize}
\item[(i)] $\vartheta \mapsto \mathcal L^N(\vartheta;X^{(N)})$ is differentiable for every $\vartheta$ in (the interior of) $\Theta$, $\mathbb W^N$-almost surely,
\item[(ii)] $\vartheta \mapsto \nabla_{\vartheta} \mathcal L^N(\vartheta; X^{(N)})$ is continuous in quadratic $\mathbb W^N$-mean, for every $\vartheta$ in (the interior of) $\Theta$,
\item[(iii)] we have finite Fisher information 
$$
\EE_{\PP_{\vartheta}^{N}} \big[|\nabla_{\vartheta} \ell^N(\vartheta;X^{(N)}) |^2\big]< \infty$$ 
for  every $\vartheta$ in (the interior of) $\Theta$.
\end{itemize}
\end{defi}
\begin{prop} \label{prop: regularity model}
Under Assumptions \ref{ass: init condition}, \ref{reg sigma}, \ref{reg b} and \ref{hyp : reg sigma b} the (sequence of) experiment(s) $(\mathcal E^N)_{N \geq 1}$ is regular.
\end{prop}
\begin{proof}[(Sketch of) Proof]
By exchanging the order of the differentiation with respect to $\vartheta$ and the stochastic integral 
we have
\begin{align*} 
\partial_{\vartheta_k}\ell^N(\vartheta;X^{(N)}) & = \sum_{i=1}^N \int_0^T \partial_{\vartheta_k} (c^{-1}b)(\vartheta; t,X^i_t,\mu^{(N)}_t)^{\top}dX^i_t \\
&- \sum_{i=1}^N\int_0^T \partial_{\vartheta_k} (c^{-1/2}b)(\vartheta; t,X^i_t,\mu^{(N)}_t)^\top (c^{-1/2}b)(\vartheta; t,X^i_t,\mu^{(N)}_t) dt.
 \end{align*}
We obtain the representation
\begin{equation} \label{eq : likelihood forme martingale}
 \partial_{\vartheta_k}\ell^N(\vartheta;X^{(N)}) = \sum_{i=1}^N \int_0^T \partial_{\vartheta_k} (c^{-1/2}b)(\vartheta; t,X^i_t,\mu^{(N)}_t)^{\top}dB^{i, N,\vartheta}_t,
\end{equation}
where the 
$$(B^{i, N, \vartheta}_t)_{t \in [0,T]} = \Big(\int_0^t c^{-1/2}(s,X_s^i)(dX_s^i-b(\vartheta; s,X_s^i,\mu_s^{(N)})ds)\Big)_{t \in [0,T]},\;\;1 \leq i \leq N$$
 are independent Brownian motions on $\R^d$ under $\PP_\vartheta^N$. The properties (i), (ii) and (iii) are then a simple consequence of Assumption \ref{hyp : reg sigma b} together with the following moment bound, 
\begin{lem} \label{lem: moment bound}
Under Assumptions \ref{ass: init condition}, \ref{reg sigma}, \ref{reg b}, for every $r \geq 1$, we have
$$\sup_{\vartheta \in \Theta, t \in [0,T], N \geq 1} \E_{\PP_\vartheta^N}[|X_t^i|^r] < \infty.$$
\end{lem} 
Note that $\E_{\PP_\vartheta^N}[|X_t^i|^r]$ does not depend on $i$. The proof of Lemma \ref{lem: moment bound} is given in Appendix \ref{app: proof of moment bound}.
\end{proof}

Finally, we have a notion of Fisher information matrix by setting 
$$
\mathbb I_{\mathcal E^N}(\vartheta)  =  \EE_{\PP_{\vartheta}^{N}} \big[\nabla_{\vartheta} \ell^N(\vartheta;X^{(N)}) \nabla_{\vartheta} \ell^N(\vartheta;X^{(N)})^{\top}\big]. 
$$
Thanks to \eqref{eq : likelihood forme martingale}, we also have
\begin{equation} \label{eq: fisher info}
\mathbb I_{\mathcal E^N}(\vartheta) = \Big(\sum_{i=1}^N \EE_{\PP_{\vartheta}^{N}}\Big[ \int_0^T  \partial_{\vartheta_\ell}(c^{-1/2}b)(\vartheta;t,X^i_t,\mu^{(N)}_t) \partial_{\vartheta_{\ell'}}(c^{-1/2}b)(\vartheta;t,X^i_t,\mu^{(N)}_t)^{\top}   dt\Big]\Big)_{1 \leq \ell, \ell' \leq p}.
 \end{equation}

\subsection{The companion McKean-Vlasov product experiment}\label{subsection : syst interacting limit}

We let $\mathcal C = \mathcal C([0,T],\R^d)$ denote the space of continuous functions on $\R^d$, equipped with the natural filtration $(\mathcal F_t)_{0 \leq t \leq T}$  induced by the canonical mapping $X_t(\omega)=\omega_t$. For every $\vartheta \in \Theta$, we let $\overline{\PP}_\vartheta$ denote the unique law under which 
the process 
$$(B_t^{\vartheta})_{t \in [0,T]} = \Big(\int_0^t c^{-1/2}(s,X_s)(dX_s-b(\vartheta; s,X_s, \mu_s^\vartheta)ds)\Big)_{t \in [0,T]}$$ is a standard Brownian motion on $\R^d$, appended with the condition $\mathcal L(X_0) = \mu_0$, and $\mu^\vartheta = (\mu_t^\vartheta)_{t \in [0,T]}$ is a probability solution of \eqref{eq: mckv approx}. The family $(\overline{\PP}_\vartheta)_{\vartheta \in \Theta}$ is well-defined under Assumptions \ref{ass: init condition}, \ref{reg sigma}, \ref{reg b}. In particular, the canonical process $X$ on $(\mathcal C, \mathcal F_T)$ is a solution to the McKean-Vlasov equation 
\begin{equation} \label{eq : diff limite}
\left\{
\begin{array}{l}
dX_t = b(\vartheta;t,X_t,\mu_t^\vartheta) dt + \sigma(t,X_t)dB_t^{\vartheta},\; t \in [0,T],\\  
\mathcal L(X_0) = \mu_0.
\end{array}
\right.
\end{equation}
The following result is the counterpart of Lemma \ref{lem: moment bound}. Note in particular that the marginals of $\overline{\PP}_\vartheta$ coincide with the solution $\mu^\vartheta = (\mu_t^\vartheta)_{t \in [0,T]}$ of the Fokker-Planck equation \eqref{eq: mckv approx}.
\begin{lem} \label{lem : puissances limite}
Under Assumptions \ref{ass: init condition}, \ref{reg sigma}, \ref{reg b}, for every $r \geq 1$, we have
$$\sup_{\vartheta \in \Theta, t \in [0,T]} \E_{\overline{\PP}_\vartheta}[|X_t|^r]  = \sup_{\vartheta \in \Theta, t \in [0,T]}\int_{\R^d}|x|^r \mu_t^{\vartheta}(dx)< \infty.$$
\end{lem}
The proof is given in Section \ref{app: proof of moment limit}. We also have the following smoothness property in the parameter $\vartheta$, proof of which is delayed until Section \ref{app: proof of smooth}.
\begin{prop} \label{prop: smoothness parameter McKean}
Under Assumption \ref{ass: init condition}, \ref{reg sigma}, \ref{reg b} and \ref{hyp : reg sigma b}, the mapping $\vartheta \mapsto \mu_t^{\vartheta}$ is Lipschitz continuous in the Wasserstein-1 metric $\mathcal W_1$, uniformly in $t \in [0,T]$.
\end{prop}
We next consider the limit experiment
$$\mathcal G = \big(\mathcal C, \mathcal F_T, (\overline{\PP}_\vartheta)_{\vartheta \in \Theta}\big)$$
and its $N$-fold counterpart
$$\mathcal G^{\otimes N} =  \Big(\mathcal C^N, \mathcal F_T^N, (\overline{\PP}_\vartheta^{\otimes N})_{\vartheta \in \Theta}\Big)$$
that serves as an approximation for the experiment $\mathcal E^N$. Inspired by classical propagation of chaos techniques (see in particular \cite{LAcker}), we can easily show that the measures $\PP_\vartheta^N$ and $\overline{\PP}_\vartheta^{\otimes N}$ are indistinguishable when the drift is of the form
\begin{equation} \label{eq: linear}
b(\vartheta; t,x,\nu) = \int_{\R^d}\widetilde b(\vartheta; t,x,y)\nu(dy),
\end{equation}
for some kernel $\widetilde b(\vartheta; \cdot): [0,T] \times \R^d \times \R^d \rightarrow \R^d$ such that 
\begin{equation} \label{eq: hyp noyau}
\sup_{t \in [0,T], \vartheta \in \Theta} \big|\widetilde b(\vartheta; t;x;y)\big| \leq C(1+|x|^{r_1}+|y|^{r_2})
\end{equation}
for some $r_1,r_2 \geq 1$, a situation that covers most of our examples, see Section \ref{sec: examples} below. More precisely, we have the following 
\begin{prop} \label{prop: indisting}
Under Assumptions \ref{ass: init condition}, \ref{reg sigma}, \ref{reg b}, if $b$ has moreover the form \eqref{eq: linear}-\eqref{eq: hyp noyau}, we have
\begin{equation} \label{eq: entropy control}
\limsup_{N \rightarrow \infty}\sup_{\vartheta \in \Theta} \E_{\overline{\PP}_{\vartheta}^{\otimes N}}\Big[\log\frac{d\overline{\PP}_\vartheta^{\otimes N}}{d\PP_\vartheta^N}\Big] < \infty.
\end{equation}
In particular, if
$$\sup_{\vartheta \in \Theta} \int_0^T \int_{\R^d \times \R^d}|\widetilde b(\vartheta; t,x,y)|^2(\mu_t^\vartheta \otimes \mu_t^\vartheta)(dx, dy) dt< 4,
$$
then
\begin{equation} \label{eq: pinsker}
\limsup_{N \rightarrow \infty}\sup_{\vartheta \in \Theta}\|\PP_{\vartheta}^N-\overline{\PP}_\vartheta^{\otimes N}\|_{TV} < 1,
\end{equation}
where $\|\cdot\|_{TV}$ denotes total variation distance. 
\end{prop}
The proof is given in Section \ref{proof: prop: indisting}. Some remarks are in order: {\bf 1)} The estimate \eqref{eq: entropy control} tells us that it is impossible to statistically discriminate between $\PP_{\vartheta}^N$ and $\overline{\PP}_\vartheta^{\otimes N}$ asymptotically. More precisely, inequality \eqref{eq: pinsker} shows in particular that  provided $\widetilde b$ is not too big or $T$ not too large, then there exists no test of the null $H_0: \PP_{\vartheta}^N = \overline{\PP}_\vartheta^{\otimes N}$ against the alternative $H_1: \PP_{\vartheta}^N \neq \overline{\PP}_\vartheta^{\otimes N}$ with asymptotically arbitrarily small first and second kind error in the limit $N \rightarrow \infty$. {\bf 2)} We will actually prove a stronger result in Section \ref{sec: main} below, showing that both $(\mathcal E^N)_{N \geq 1}$ and $(\mathcal G^{\otimes N})_{N \geq 1}$ share the LAN property, with same asymptotic Fisher information. {\bf 3)} Finally, \eqref{eq: entropy control} may hold in wider generality when the dependence in the measure variable in the drift is nonlinear, as soon as we have some differentiability in the following sense: there exists $\partial_\nu b(\vartheta;t,x,\cdot): \R^d \times \mathcal P_1  \rightarrow \R^d$ such that
$$b(\vartheta;t,x,\nu)-b(\vartheta;t,x,\nu') = \int_0^1 \partial_\nu b(\vartheta;t,x,y,\lambda \nu+(1-\lambda) \nu') (\nu-\nu')(dy)$$
for every $\nu,\nu' \in \mathcal P_1$ and $\partial_\nu b(\vartheta;t,x,\cdot)$ satisfies additional smoothness properties. Iterating the operator $\partial_\nu$, if 
$\partial_\nu^k b(\vartheta;t,x,\cdot): (\R^d)^k \times \mathcal P_1 \rightarrow \R^d$ exists and satisfies some smoothness and integrability properties, we may expect \eqref{eq: entropy control} to hold as soon as $k \geq d/2$. We refer to Assumption 4 and Proposition 19 of \cite{della2022nonparametric} where this approach is developed.\\

We also have a log-likelihood in the experiment $\mathcal G^{\otimes N}$ by setting
\begin{equation}\label{log-likelihood Girsanov product}
\overline{\ell}^N(\vartheta; X^{(N)}) = \sum_{i=1}^N  \int_0^T  (c^{-1}b)(\vartheta; t,X^{i}_t,\mu^{\vartheta}_t)^{\top} dX^i_t
-\frac{1}{2} \sum_{i=1}^N \int_0^T |(c^{-1/2}b)(\vartheta; t,X^{i}_t,\mu^{\vartheta}_t)|^2 dt. 
\end{equation}
This is the same argument as before:  the laws $\overline{\PP}_\vartheta^{\otimes N}$ are all absolutely continuous w.r.t. $\mathbb W^N$, and for every $\vartheta \in \Theta$,
$$\frac{d\overline{\PP}_{\vartheta}^{\otimes N}}{d\mathbb W^N}(X^{(N)}) = \exp\big(\overline{\ell}^N(\vartheta; X^{(N)})\big)$$
holds 
$\mathbb W^N$-almost-surely.\\ 
Finally under Assumptions  \ref{ass: init condition}, \ref{reg sigma}, \ref{reg b} and \ref{hyp : reg sigma b}, the (sequence of) experiment(s)  $\mathcal{G}^{\otimes N}$ is also a regular model and its (normalised) Fisher information $\mathbb I_{\mathcal G}(\vartheta) = N^{-1}\mathbb I_{\mathcal G^{\otimes N}}(\vartheta)$ is given by
\begin{align*}
   &   N^{-1}\EE_{\overline{\PP}_\vartheta} \big[\nabla_{\vartheta} \overline{\ell}^N(\vartheta; X^{(N)}) \nabla_{\vartheta} \overline{\ell}^N(\vartheta; X^{(N)})^{\top}\big] \\
&= \sum_{j=1}^d \int_0^T \int_{\RR^d}  \nabla_{\vartheta} (c^{-1/2}b)^j(\vartheta;t,x,\mu^{\vartheta}_t)\nabla_{\vartheta} (c^{-1/2}b)^j(\vartheta;t,x,\mu^{\vartheta}_t)^{\top}  \mu^{\vartheta}_t(dx)dt \\
& =  \Big(\int_0^T \int_{\R^d} \partial_{\vartheta_\ell}(c^{-1/2}b)(\vartheta;t,x,\mu^{\vartheta}_t) \partial_{\vartheta_{\ell'}}(c^{-1/2}b)(\vartheta;t,x,\mu^{\vartheta}_t)^{\top}  \mu_t^\vartheta(dx) dt\Big)_{1 \leq \ell, \ell' \leq p}.
\end{align*}
Moreover, the mapping $\vartheta \mapsto \mathbb I_{\mathcal G}(\vartheta)$ is smooth and appears as the (normalised) asymptotic information of $\mathcal E^N$:
\begin{prop} \label{prop: fisher cont}
Under Assumptions  \ref{ass: init condition}, \ref{reg sigma}, \ref{reg b} and \ref{hyp : reg sigma b}, the mapping $\vartheta \mapsto \mathbb I_{\mathcal G}(\vartheta)$ is Lipschitz continuous. Moreover, for every $\vartheta$ in (the interior of) $\Theta$, we have
$$N^{-1} \mathbb I_{\mathcal E^N}(\vartheta) \rightarrow \mathbb I_{\mathcal G}(\vartheta)$$
as $N \rightarrow \infty$, where $ \mathbb I_{\mathcal E^N}(\vartheta)$ is the Fisher information matrix of the experiment $\mathcal E^N$ defined in \eqref{eq: fisher info} above.
\end{prop}
The proof is given in Section \ref{app: proof prop: fisher cont}.

\subsection{Identifiability and non-degeneracy of the Fisher information}  \label{sec: identif and non-deg}

\subsubsection*{Motivation}

In the preceding section, we have built $\mathcal E^N$ and $\mathcal G^{\otimes N}$ (equivalently $\mathcal G$) as possibly redundant, in the sense that the mappings $\vartheta \mapsto \PP_\vartheta^N$ and $\vartheta \mapsto \overline{\PP}_\vartheta$ are not necessarily one-to-one on $\Theta$. Having a well-posed parametrisation is required since we wish to have at least consistent estimators. Arguing asymptotically, we only need to work in the limit model $\mathcal G$.\\ 

Also, asymptotic identifiability is somehow linked to the non-degeneracy of the (normalised) Fisher information matrix $\mathbb I_{\mathcal G}$. Following 
\cite{rothenberg1971identification}, see also  \cite{Tse1973}, we say that a point $\vartheta$ in (the interior of) $\Theta$ is {\it regular} if $\vartheta' \mapsto \mathbb I_{\mathcal G}(\vartheta')$ has constant rank in a neighbourhood of $\vartheta$ and the experiment $\mathcal G$ is called {\it locally identifiable} at $\vartheta$ if the mapping $\vartheta' \mapsto \overline{\PP}_{\vartheta'}$ is injective in a neighbourhood of $\vartheta$. We have the following classical result (that goes back at least to Cramer \cite{cramer1946mathematical}):

\begin{prop}[Theorem 1 in \cite{rothenberg1971identification}]  \label{prop: tse}
If $\vartheta$ is regular, then $\mathcal G$ is locally identifiable at $\vartheta$ if and only if $\mathbb I_{\mathcal G}(\vartheta)$ has full rank.
\end{prop}

Unfortunately, there is no hope to obtain a global result that links the two notions unless in very specific cases, see Proposition \ref{prop: eq-identif-nondeg} below. We next give {\it ad-hoc} assumptions that give sufficient (and independent) condition for both identifiability and non-degeneracy of the Fisher information. 

\subsubsection*{An identifiability assumption}

We first have a relatively weak assumption that guarantees global identifiability in 
$\mathcal G$.

\begin{assumption}  \label{ass: identif product experiment}
For all $\vartheta \in \Theta$, for $\overline{\PP}_\vartheta$-almost all $\omega$, for all $\vartheta' \neq \vartheta$, the func\-tions $t \mapsto b(\vartheta; t,X_t(\omega),\mu^{\vartheta}_t)$ and $t \mapsto b(\vartheta'; t,X_t(\omega),\mu^{\vartheta'}_t)$ are not $dt$-a.e. equal. 
\end{assumption}

Assumption 
\ref{ass: identif product experiment} is relatively standard in the literature of statistics of random processes and minimal (see {\it e.g.} \cite{genon1993estimation} in a somewhat analogous context). Indeed, by Girsanov's theorem, for two different parameters $\vartheta, \vartheta' \in \Theta$, the laws $\overline{\PP}_\vartheta$ and $\overline{\PP}_{\vartheta'}$ are absolutely continuous and 
\begin{align*}
\log \frac{d\overline{\PP}_\vartheta}{d\overline{\PP}_{\vartheta'}}(X)&  =  \int_0^T  ((c^{-1}b)(\vartheta; s,X_s,\mu^{\vartheta}_s)- (c^{-1}b)(\vartheta'; s,X_s,\mu^{\vartheta}_s))^{\top} dX^i_s\\
&-\frac{1}{2} \int_0^T (|(c^{-1/2}b)(\vartheta; s,X^{i}_s,\mu^{\vartheta}_s)|^2 - |(c^{-1/2}b)(\vartheta'; s,X^{i}_s,\mu^{\vartheta}_s)|^2)ds.
\end{align*}
Having Assumption \ref{ass: identif product experiment} fail for some $\vartheta'$ implies $\overline{\PP}_{\vartheta}\big(\frac{d\overline{\PP}_\vartheta}{d\overline{\PP}_{\vartheta'}}(X)=1\big)$, {\it i.e.} $\overline{\PP}_{\vartheta}=\overline{\PP}_{\vartheta'}$. Assumption \ref{ass: identif product experiment} may be difficult to check in practice. Yet, it is satisfied as soon as the mapping $\vartheta \mapsto \big((t,x) \mapsto b(\vartheta;t,x, \mu_t^\vartheta)\big)$ is one-to-one. Also, for certain form of the likelihood, we have other criteria, see Proposition \ref{prop: eq-identif-nondeg} below.

\subsubsection*{Non-degeneracy of the information} We need some notation. For any $\vartheta,\vartheta' \in \Theta$ such that the segment $[\vartheta,\vartheta'] =\{\vartheta+\lambda(\vartheta'-\vartheta), \lambda \in [0,1]\} \subset \Theta$ and a function $\phi$ defined on $\Theta$, we set
$$\phi([\vartheta, \vartheta']) = \int_0^1\phi(\vartheta+\lambda(\vartheta'-\vartheta))d\lambda.$$ 
\begin{defi} \label{def: non deg}
The statistical experiment $\mathcal G$ is non-degenerate if
\begin{equation} \label{eq: strong nondeg}
  \inf_{[\vartheta,\vartheta'] \subset\Theta} \mathsf{det}\,\EE_{\overline{\PP}_\vartheta} \big[\nabla_{\vartheta} \overline{\ell}^1([\vartheta, \vartheta']) \nabla_{\vartheta} \overline{\ell}^1([\vartheta, \vartheta'])^{\top}\big] >0,
\end{equation}
where $ \mathsf{det}$ denotes the determinant.
\end{defi}
Equivalently, we can rewrite \eqref{eq: strong nondeg} as 
$$ \inf \mathsf{det}\,\Big(\sum_{j=1}^d\; \int_0^T \int_{\RR^d} \nabla_{\vartheta} (c^{-1/2}b)^j([\vartheta,\vartheta'];t,x,\mu^{\vartheta}_t)\nabla_{\vartheta} (c^{-1/2}b)^j([\vartheta, \vartheta'];t,x,\mu^{\vartheta}_t)^{\top}\mu_t^\vartheta(dx)dt\Big) >0,
$$ 
where the infimum is taken over all segments $[\vartheta,\vartheta'] \subset \Theta$.
%
Obviously, if $\mathcal G$ is non-degenerate,  taking $\vartheta = \vartheta'$, 
Definition \ref{def: non deg} boils down to 
\begin{equation} \label{eq: weak non deg}
\inf_{\vartheta \in \Theta} \mathsf{det}\,\mathbb I_{\mathcal G}(\vartheta) >0
\end{equation}
{\it i.e.} $\vartheta \mapsto \mathbb I_{\mathcal G}(\vartheta) $ has full rank uniformly in $\vartheta$
and we find back the usual non-degeneracy of the Fisher information. The somewhat stronger non-degeneracy criterion that we pick in Definition \ref{def: non deg} enables us to check the assumptions of the theory of Ibragimov and Hasminski for obtaining sharp properties for the maximum likelihood estimator (see in particular Step 2 of the proof of Theorem \ref{thm: MLE prop} in Section \ref{sec: proof of theorem mle} below). In explicit examples, proving  \eqref{eq: strong nondeg} is no more difficult than proving \eqref{eq: weak non deg}, see Section \ref{sec: examples} below.

\subsubsection*{Checking \eqref{eq: strong nondeg} or \eqref{eq: weak non deg} in practice}


A special difficulty for the statistical analysis of  $\mathcal E^N$ or  rather $\mathcal G$ lies in the asymptotic form \eqref{eq : diff limite} with the presence of $(\mu_t^\vartheta)_{0 \leq t \leq T}$ in the drift, which is never explicit, except in very special cases with a specific moment structure in the measure dependence, see Section \ref{sec: examples} below.\\

It is noteworthy that 
\eqref{eq: strong nondeg} can usually be tested in a simple way given an explicit parametrisation. Indeed, 
Definition \ref{def: non deg} is equivalent to show that for every segment $[\vartheta, \vartheta'] \subset \Theta$, 
$$ 
\inf_{[\vartheta, \vartheta'] \subset \Theta}\min_{|z| = 1}  \sum_{j=1}^d \int_0^T \; \int_{\R^d}(\nabla_{\vartheta} \big(c^{-1/2}b)^j([\vartheta, \vartheta'];t,x,\mu^{\vartheta}_t)^{\top} z\big)^2 \mu^{\vartheta}_t(dx) dt >0.$$
Under Assumptions  \ref{ass: init condition}, \ref{reg sigma}, \ref{reg b}, we have that $\mu_t^\vartheta(dx) = \mu_t^\vartheta(x)dx$ is absolutely continuous on $\R^d$ for $t>0$, and we may pick a version $\mu_t^\vartheta$ of the density that is continuous and positive on $\R^d$. This follows from classical Gaussian tail estimates for the solution of parabolic equations. We refer for example to Corollary 8.2.2 of \cite{BoKry}. 
By a simple continuity argument,
it is then sufficient to show that there cannot exist a segment $[\vartheta, \vartheta'] \subset \Theta$ and some $|z|=1$, such that the function
$$x \mapsto   \int_0^T \sum_{j = 1}^d\big(\nabla_{\vartheta} (c^{-1/2}b)^j([\vartheta, \vartheta'];t,x,\mu^{\vartheta}_t)^{\top} z\big)^2dt$$
vanishes asymptotically, or, as soon as we have continuity in $t$ as $t \rightarrow 0$, if one of the functions
$$x \mapsto  \nabla_{\vartheta}(c^{-1/2}b)^j([\vartheta, \vartheta'];0,x,\mu_0)^{\top} z,\;\;j = 1,\ldots, d$$
does not identically vanishes. This last criterion has the advantage to avoid the term $\mu_t^\vartheta$ for $t>0$. We gather these observations in the following:


\begin{prop} \label{prop: criterion fisher}
Work under Assumptions \ref{ass: init condition}, \ref{reg sigma}, \ref{reg b} and \ref{hyp : reg sigma b}. 
Assume moreover
that the functions
$$t \mapsto \nabla_{\vartheta} (c^{-1/2}b)^j([\vartheta, \vartheta'];t,x,\mu^{\vartheta}_t),\;\;j = 1,\ldots, d$$
are all continuous at $t = 0$ for every $[\vartheta, \vartheta'] \subset \Theta$ and a.e.-almost $x \in \R^d$.\\ 

If, for every $[\vartheta, \vartheta'] \subset \Theta$ and any $z\in \R^p$ with $|z| = 1$, one of the functions 
\begin{equation} \label{eq: continuity}
x \mapsto  \nabla_{\vartheta}(c^{-1/2}b)^j([\vartheta, \vartheta'];0,x,\mu_0)^{\top} z,\;\;\;j = 1,\ldots, d
\end{equation}
does not identically vanishes, then $\mathcal G$ is non-degenerate in the sense of Definition \ref{def: non deg}.
\end{prop}

We specifically apply this criterion in the examples Section \ref{sec: examples} and check that the criterion \eqref{eq: continuity} is particularly simple to establish when the dependence in the measure argument of the function $b$ is of the form \eqref{eq: linear}.

\subsubsection*{A case of equivalence between global identifiability and non-degeneracy of the information}
We revisit Theorem 3 in \cite{rothenberg1971identification} to obtain the following criterion:
\begin{prop} \label{prop: eq-identif-nondeg}
Work under Assumptions \ref{ass: init condition}, \ref{reg sigma}, \ref{reg b} and \ref{hyp : reg sigma b}.
Assume that the log-likelihood $\ell^N(\vartheta; X^{(N)})$ in $\mathcal E^N$ defined by \eqref{log-likelihood Girsanov}  has the form
\begin{equation} \label{eq: rep exp}
\ell^N(\vartheta,; X^N)  = \vartheta^\top G^N(X^{(N)}) + \vartheta^\top H^N(X^{(N)}) \vartheta,
\end{equation}
where $G^N$ and $H^N$ are functions of the trajectory $X^{(N)}$ with values in $\R^p$ and $\R^p \otimes \R^p$ respectively, and $(H^N)^\top = H^N$ is symmetric. If $\Theta_0 \subset \Theta$ is a convex set such that  $\mathbb I_{\mathcal G}(\vartheta)$ is non-singular for every $\vartheta \in \Theta_0$, then, both $(\mathcal E^N)_{N \geq 1}$ and $\mathcal G$ are identifiable on $\Theta_0$.
\end{prop}

By identifiability of the sequence of experiment $(\mathcal E^N)_{N \geq 1}$, we mean injectivity of the mapping $\vartheta \mapsto (\PP_\vartheta^N)_{N \geq 1}$ ({\it i.e.} simultaneously for every $N \geq 1$).
The proof is given in Section \ref{proof prop: eq-identif-nondeg}. In the specific case of McKean type models that date back to \cite{MCKEAN, Sznitman1, tanaka1984limit} and widely used in practice (see e.g. \cite{carmona2018probabilistic, fouque2013systemic} or \cite{kasonga1990maximum} in statistics), we have in some instances a representation like \eqref{eq: rep exp} and explicit formulas for $\mathbb I_{\mathcal G}(\vartheta)$, which gives global identifiability for free as soon as $\mathbb I_{\mathcal G}(\vartheta)$ is non-degenerate.  See the examples in Section \ref{sec: examples}.

%

\section{Main results} \label{sec: main}

\subsection{The LAN property} The local asymptotic normality property of a statistical model characterises its regularity: it expresses the fact that the experiment locally resembles a Gaussian shift in an optimal scale driven by the Fisher information. It has powerful consequences in terms of properties of optimal procedures via the celebrated H\'ajek convolution theorem \cite{hajek1972local}. More precisely the sequence of experiments  $(\mathcal E_N)_{N \geq 1}$ satisfies the LAN property at $\vartheta \in \Theta$ with information rate $N\mathbb I_{\mathcal G}(\vartheta)$ if
\begin{equation} \label{eq: LAN}
\log \frac{d\PP_{\vartheta+(N\mathbb I_\mathcal G(\vartheta))^{-1/2} u}^N}{d\PP_{\vartheta}^N} = u^\top \xi_{\vartheta}^N-\tfrac{1}{2}|u|^2+r_N(\vartheta, u),
\end{equation}
where $\xi_{\vartheta}^N$ converges in distribution under $\PP_\vartheta^N$ to standard Gaussian variable in $\R^p$ and $r_N(\vartheta, u) \rightarrow 0$ in $\PP_\vartheta^N$-probability. Of course, the convergence \eqref{eq: LAN} is meaningful only if $\vartheta+(N\mathbb I_{\mathcal G}(\vartheta))^{-1/2} u \in \Theta$ and is well-defined, {\it i.e.} if $\mathsf{det}\,\mathbb I_\mathcal G(\vartheta) >0$. This is granted for instance for $\vartheta$ in the interior of $\Theta$ for large enough $N$ and under \eqref{eq: weak non deg}. 

\begin{thm} \label{thm: LAN prop}
Work under Assumptions \ref{ass: init condition}, \ref{reg sigma}, \ref{reg b}, \ref{hyp : reg sigma b} and \ref{ass: identif product experiment}. Assume moreover that $\mathcal G$ is non-degenerate according to Definition \ref{def: non deg}. For every $\vartheta$ in (the interior of) $\Theta$, the sequence of experiments  $(\mathcal E^N)_{N \geq 1}$ is locally asymptotically normal at $\vartheta$ with information rate $N\mathbb I_{\mathcal G}(\vartheta)$. 

The same result holds for $(\mathcal G^{\otimes N})_{N \geq 1}$. 
\end{thm}

Several remarks are in order: {\bf 1)} Theorem \ref{thm: LAN prop} is the most powerful result one can obtain about the structure of $(\mathcal E^N)_{N \geq 1}$ and $(\mathcal G^{\otimes N})_{N \geq 1}$: it tells us that around a given point $\vartheta_0$, if we parametrise locally the experiment via $\vartheta = \vartheta_0 + N^{-1/2}u$ with $u \in \R^p$ being the unknown parameter, then the experiments look like the simplest possible experiment, namely  a Gaussian shift 
$$Y^N = u + \mathbb I_{\mathcal G}(\vartheta_0)^{-1/2}\xi + o(1)$$
where $\xi$ is a standard normal $\mathcal N(0,\mathrm{Id}_{\R^p})$ and $o(1)$ is a small term that vanishes in $\PP_{\vartheta}^N$ or $\overline{\PP}_{\vartheta}^{\otimes N}$ probability, locally uniformly in $u$. {\bf 2)} The fact that both $(\mathcal E^N)_{N \geq 1}$ and $(\mathcal G^{\otimes N})_{N \geq 1}$ share the LAN property with same asymptotic Fisher variance quantifies their asymptotic similarity, see in particular Proposition \ref{prop: indisting}. {\bf 3)} The LAN property has several consequences in terms of strong properties of the maximum likelihood estimator, see Theorem \ref{thm: MLE prop} below. In particular, the first simple consequence is given in terms of exact asymptotic minimax lower bounds: call a centrally symmetric function $w: \R^p \rightarrow [0,\infty)$  such that the sets $\{w < c\}, c>0$ are all convex a {\it polynomial loss function} if it admits a polynomial majorant. 

\begin{cor} \label{cor strong LAN}
In the setting of Theorem \ref{thm: LAN prop}, let $w$ be a polynomial loss function. Then, for any estimator $\widehat \vartheta_N$ in $\mathcal E^N$ and any sufficiently small $\delta >0$, for every $\vartheta$ in (the interior of) $\Theta$ for which $\mathsf{det}\,\mathbb I_{\mathcal G}(\vartheta) >0$, we have
$$\liminf_{N \rightarrow \infty}\sup_{|\vartheta'-\vartheta| \leq \delta} \E_{\PP_{\vartheta'}^N}\big[w\big(N^{1/2}\mathbb I_{\mathcal G}(\vartheta)^{1/2}(\widehat \vartheta_N-\vartheta')\big)\big] \geq (2\pi)^{-p/2}\int_{\R^p} w(x)\exp(-\tfrac{1}{2}|x|^2) dx.$$
The same result holds true for $\widehat \vartheta_N$ in $\mathcal G^{\otimes N}$ replacing $\PP_\vartheta^N$ by $\overline{\mathbb P}_\vartheta^{\otimes N}$.
\end{cor}

Corollary \ref{cor strong LAN} is a simple application of H\'ajek convolution theorem, given the LAN property of Theorem \ref{thm: LAN prop}, see {\it e.g.} Theorem II.12.1 (an in particular Remark III.12.1) in \cite{ibragimov2013statistical}. It provide with a sharp local asymptotically minimax bound, up to constants. We shall see below that the maximum likelihood estimator achieves this bound.

\subsection{Maximum likelihood estimation and properties}

We elaborate on the properties of the maximum likelihood estimator by relying on (a uniform version of) the LAN property of Theorem \ref{thm: LAN prop}. It implies several fine results that go beyond the usual asymptotic weak expansions given by an {\it ad-hoc} study of the form of the estimator, as is usually the case in the literature.

\begin{thm} \label{thm: MLE prop}
Work under Assumptions \ref{ass: init condition}, \ref{reg sigma}, \ref{reg b}, \ref{hyp : reg sigma b} and \ref{ass: identif product experiment}. Then, for large enough $N$, the solution $\widehat \vartheta_N^{\,\tt{mle}}$ to
\begin{equation} \label{eq: def emv}
\mathcal L^N(\widehat \vartheta_N^{\,\tt{mle}};X^{(N)}) = \sup_{\vartheta \in \Theta} \mathcal L^N(\vartheta; X^{(N)})
\end{equation}
is well-defined. Moreover, the following asymptotic upper bounds are valid:
\begin{enumerate}
\item[(i)] if $\mathcal G$ is non-degenerate in the sense of Definition \ref{def: non deg},
$$\sqrt{N}\big(\widehat \vartheta_N^{\,\tt{mle}}-\vartheta\big) \rightarrow \mathcal N\big(0, \mathbb I_{\mathcal G}(\vartheta)^{-1}\big)$$
in $\PP_\vartheta^N$-distribution as $N \rightarrow \infty$. 
\item[(ii)] For every polynomial loss function $w$ and any $\vartheta$ in the interior of $\Theta$, we have exact local asymptotic minimax optimality:
$$\limsup_{N \rightarrow \infty}\sup_{|\vartheta'-\vartheta| \leq \delta}\E_{\mathbb \PP_{\vartheta'}^N}\big[w\big(N^{1/2}\mathbb I_{\mathcal G}(\vartheta)^{1/2}\big(\widehat \vartheta_N^{\,\tt{mle}}-\vartheta'\big)\big)\big] \rightarrow (2\pi)^{-p/2}\int_{\R^p} w(x)\exp(-\tfrac{1}{2}|x|^2) dx$$
as $\delta \rightarrow 0$. 
\item[(iii)] For every polynomial loss function $w$ and any (non empty) open set $\Theta_0 \subset \Theta$, we have global asymptotic minimax optimality:
$$\mathcal R_w^N(\widehat \vartheta_N^{\,\tt{mle}}; \Theta_0)= \inf_{\widehat \vartheta_N}\mathcal R_w^N(\widehat \vartheta_N; \Theta_0)(1+o(1))$$
as $N \rightarrow \infty$, where
$$\mathcal R_w^N(\widehat \vartheta_N; \Theta_0) = \sup_{\vartheta \in \Theta_0}\E_{\PP_\vartheta^N}\big[w\big(N^{1/2}\mathbb I_{\mathcal G}(\vartheta)^{1/2}(\widehat \vartheta_N-\vartheta)\big)\big].$$
\end{enumerate}
\end{thm}

Some further remarks: {\bf 1)} We find back the classical asymptotic properties (i) of the maximum likelihood estimator that are given in the literature, but the result is appended by a much stronger convergence in (ii), that matches in particular the lower bound of Corollary \ref{cor strong LAN}. {\bf 2)} We finally obtain global asymptotic minimax optimality by (iii), which is the parametric analog (in a much more precise way) of our minimax results of Section 4 in \cite{della2022nonparametric} in the nonparametric case.

\section{Examples} \label{sec: examples}

In this section, we elaborate on specific examples that appear in the literature and in applications. We first revisit the linear McKean model studied at length in \cite{kasonga1990maximum}. We slightly extend in Section \ref{sec: kasonga simple} his example (1.3) from $p=2$ to $p=3$. In Section \ref{sec: kasonga extended}, we develop an example of a generalised linear form and show in particular how our identifiability  and non-degeneracy criteria of Section \ref{sec: identif and non-deg} are easily implementable and avoid to use the machinery of \cite{kasonga1990maximum}. In Section \ref{sec: swarming}, we develop a non-trivial example of kinetik mean-field model with a double layer potential that may serve in many applications, like swarming models or more general individual based-models, see \cite{bolley2011stochastic} and the references therein. We finally develop a genuinely non-linear example, {\it i.e.} when the measure argument is not linear like in \eqref{eq: explain drift}, as for instance in the examples of \cite{oelschlager1985law}. Assumption \ref{ass: init condition} is in force throughout.

\subsection{McKean-like models} \label{sec: kasonga simple}
In many applications,  \eqref{eq:diff basique} takes the explicit form 
\begin{equation} \label{eq: OU}
dX_t^i = (\vartheta_1X_t^i+\vartheta_2)dt-\vartheta_3 N^{-1}\sum_{j = 1}^N(X_t^i-X_t ^j)dt+dB_t^i,\;\;i = 1,\ldots, N
\end{equation}
with $X_t^i \in \R$. The parameter is $\vartheta = (\vartheta_1\; \vartheta_2\; \vartheta_3)^\top$. In \cite{kasonga1990maximum} the case $\vartheta_2 = 0$ is studied at length in particular. In our setting, we can encompass a more general situation with $X_t^i \in \R^d$ for some arbitrary $d \geq 1$ and replace $\vartheta_3$ by a parameter in $\R^d \otimes \R^d$ as well as $\vartheta_2$ by a parameter in $\R^d$. In this case, Assumptions  \ref{reg sigma}, \ref{reg b} and \ref{hyp : reg sigma b} are readily checked. Likewise, the identifiability and non-degeneracy assumptions can be obtained with some extra care on the initial condition. We elaborate on a specific case below.

\subsubsection*{Likelihood equations} To keep-up with notational simplicity, we detail the case $p=3$ with $\vartheta = (\vartheta_1\; \vartheta_2\; \vartheta_3)^\top \in \Theta$ as a compact subset of $\R^3$ for an ambient dimension $d=1$, with $\vartheta_1 \neq \vartheta_3$ and $\vartheta_1 \neq 0$. Introduce
$$\mathsf A_t^N(x) =
\left(
\begin{array}{ccc}
x^2 & 
x& - \langle \cdot -x,\mu_t^{(N)}\rangle^2  \\
x & 1 & 0 \\
 - \langle \cdot -x,\mu_t^{(N)}\rangle^2  &  0 &  \langle \cdot -x,\mu_t^{(N)}\rangle^2 \\
 \end{array}
\right),\;\;
\mathsf B_t^N(x) =\left(
\begin{array}{c}
x  \\
1 \\
 \langle \cdot -x,\mu_t^{(N)}\rangle   
  \end{array}
\right),
$$
where we use the bracket notation $\langle \cdot, \nu \rangle$ to denote integration w.r.t. the measure $\nu$. 
Define
\begin{equation} \label{eq: integ A}
\mathsf A_T^N = \int_0^T \langle \mathsf A_t^N(x),\mu_t^{(N)}\rangle dt
\end{equation}
and 
\begin{equation} \label{eq: integ B}
\mathsf B_T^N = N^{-1}\sum_{i =1}^N\int_0^T \mathsf B_t^N(X_t^i)dX_t^i.
\end{equation}
Thanks to the linearity in $\vartheta$ of the drift $b(\vartheta;t,x,\nu) = \vartheta_1 x + \vartheta_2 -\vartheta_3\int_{\R}(x-y)\nu(dy)$, the likelihood equations are explicit and the maximum likelihood estimator $\widehat \vartheta_N^{\,\tt{mle}}$ solves 
\begin{equation} \label{eq: mle equation}
\mathsf A_T^N \widehat \vartheta_N^{\,\tt{mle}} \,= \mathsf B_T^N.
\end{equation}
Moreover, the Fisher information matrix is given by
$$\mathbb I_{\mathcal G}(\vartheta) = \int_0^T \langle \mathsf A_t(\vartheta; x),\mu_t^\vartheta \rangle dt,$$
with
$$\mathsf A_t(\vartheta; x) =
\left(
\begin{array}{ccc}
x^2 & x & - \langle \cdot -x,\mu_t^{\vartheta}\rangle^2 \\
x & 1 & 0 \\
 - \langle \cdot -x,\mu_t^{\vartheta}\rangle^2  &  0 &  \langle \cdot -x,\mu_t^{\vartheta}\rangle^2 
 \end{array}
\right).
$$
\subsubsection*{Non-degeneracy and identifiability} The property $\mathsf{det}\,\mathbb I_{\mathcal G}(\vartheta) >0$ can be verified on the explicit form of its matrix: 
\begin{align*}
\mathsf{det}\,\mathbb I_{\mathcal G}(\vartheta)  & = \int_0^T\mathsf{Var}(\mu_t^\vartheta)dt \,\Big(T\int_0^T \overline{\mathfrak m}_1(\mu_t^\vartheta)^2dt-\big(\int_0^T\overline{\mathfrak m}_1(\mu_t^\vartheta)dt\big)^2\Big)
\end{align*}
with
$\mathsf{Var}(\nu) = \int_{\R}(x-\overline{\mathfrak m}_1(\nu))^2\nu(dx)$ and $\overline{\mathfrak m}_1(\nu)=\int_{\R^d}x\nu(dx)$. Therefore $\mathsf{det}\,\mathbb I_{\mathcal G}(\vartheta) >0$ unless $\mu_t^\vartheta$ is degenerate for all $t$ or stationary. 
In the case of a linear equation of the type \eqref{eq: OU}, the accompanying limiting measure $\mu_t^{\vartheta}$ associated to the McKean-Vlasov equation
$$dX_t = (\vartheta_1X_t+\vartheta_2)dt-\vartheta_3 (X_t-\E_{\overline{\PP}_\vartheta}[X_t])dt+dB_t$$
is a Gaussian process that can be thought of as an inhomogeneous Ornstein-Uhlenbeck model for which we have closed-form moment formulas:
\begin{equation} \label{eq: close 1}
\overline{\mathfrak m}_1(\mu_t^\vartheta) = -\vartheta_1^{-1}\vartheta_2+(\overline{\mathfrak m}_1(\mu_0)+\vartheta_1^{-1}\vartheta_2)\exp(\vartheta_1 t)
\end{equation}
and
\begin{equation} \label{eq: close 2}
\mathfrak m_2(\mu_t^\vartheta) = \exp\big(-2(\vartheta_1-\vartheta_3)t\big)\mathrm{Var}(\mu_0)+ 
\frac{1-\exp(-2(\vartheta_1-\vartheta_3)t)}{2(\vartheta_1-\vartheta_3)}
+\overline{\mathfrak m}_1(\mu_t^\vartheta)^2.
\end{equation}
In particular, having
\begin{equation} \label{eq: constraint}
\overline{\mathfrak m}_1(\mu_0)+\vartheta_1^{-1}\vartheta_2 \neq 0,\;\;\vartheta_1 \neq 0,\;\;\vartheta_1 \neq \vartheta_3
\end{equation} 
yields the non-degeneracy of $\mathbb I_{\mathcal G}(\vartheta)$ as well as the non-degeneracy in the sense of Definition \ref{def: non deg}, since $\nabla_\vartheta b(\vartheta;t,x,\nu)$ does not depend on $\vartheta$. 
Also,  
the convergence $\mathsf A_T^N \rightarrow \mathbb I_{\mathcal G}(\vartheta)$ in $\PP_\vartheta^N$-probability as $N \rightarrow \infty$ tells us that \eqref{eq: mle equation} has a well defined and unique solution with $\PP_\vartheta^N$ probability tending to one as $N \rightarrow \infty$. This last statement can be quantified via the convergence Lemma \ref{lem: conv prob} below.
Writing $\phi(x,\nu) = (x\;\;1\;\;-\int_{\R}(x-y)\nu(dy))$, the log-likelihood 
\begin{equation} \label{eq: quadrat structure}
\ell^N(\vartheta; X^{(N)}) = \sum_{i = 1}^N\int_0^T \vartheta^\top \phi(X_t ^i,\mu_t^{(N)})dX_t^i - \frac{1}{2}\sum_{i = 1}^N\int_0^T \big(\vartheta^\top \phi(X_t ^i,\mu_t^{(N)})\big)^2dt 
\end{equation} 
has representation \eqref{eq: rep exp} with
$$G^N(X^{(N)}) =  \sum_{i = 1}^N\int_0^T \phi(X_t ^i,\mu_t^{(N)})dX_t^i\;\;\text{and}\;\; H^N(X^{(N)}) =  -\frac{1}{2}\sum_{i = 1}^N\int_0^T \phi(X_t ^i,\mu_t^{(N)})\phi(X_t ^i,\mu_t^{(N)})^\top dt$$
and we obtain the identifiability of $\mathcal E^N$ and $\mathcal G$ over compact parameter sets $\Theta \subset \R^3$ that are moreover convex and satisfy the constraint \eqref{eq: constraint} by Proposition \ref{prop: eq-identif-nondeg}. Finally, explicit formulas for $\mathbb I_{\mathcal G}(\vartheta)$ and its inverse can be derived thanks to \eqref{eq: close 1} and \eqref{eq: close 2}.

\subsection{Generalised linear like models} \label{sec: kasonga extended}

We push further the preceding linear structure by considering the following model
$$dX_t^i = \vartheta_1 f(X_t^i)dt+\vartheta_2 N^{-1}\sum_{j = 1}^N g(X_t ^i-X_t^j)dt+dB_t^i,\;\;1 \leq i \leq N$$
where $f$ and $g$ are known and Lipschitz continuous real-valued functions and $X_t^i \in \R$ for simplicity. The parameter is $\vartheta  = (\vartheta_1, \vartheta_2) \in \R^2$. Thus $p=2$ which again yields simple and explicit formulas. Assumptions  \ref{reg sigma}, \ref{reg b} and \ref{hyp : reg sigma b} are readily checked. Writing $\phi(x,\nu) = (f(x)\;g\star \nu(x))^\top$, with $g\star \nu(x) =  \int_{\R}g(x-y)\nu(dy)$ the log-likelihood function
$\ell^N(\vartheta; X^{(N)})$ has again representation \eqref{eq: quadrat structure}
hence  \eqref{eq: rep exp} holds as well
and we obtain identifiability of $\mathcal E^N$ and $\mathcal G$ over compact parameter sets $\Theta \subset \R^2$ that are moreover convex 
as soon as $\mathsf{det}\, \mathbb I_{\mathcal G}(\vartheta) >0$ by Proposition \ref{prop: eq-identif-nondeg}.\\

\subsubsection*{Non-degeneracy} For proving non-degeneracy (that implies in particular $\mathsf{det}\, \mathbb I_{\mathcal G}(\vartheta) >0$), we plan to apply Proposition \ref{prop: criterion fisher}. We first notice that  $\nabla_\vartheta b(\vartheta;t,x,\nu) = \phi(x,\nu)$ does not depend on $\vartheta$, hence 
$\nabla_{\vartheta}b([\vartheta, \vartheta'];t,x,\nu) = (f(x)\; g\star \nu(x))^\top$ does not depend on the segment $[\vartheta, \vartheta']$ either. The continuity of $t \mapsto  (f(x)\; g\star \mu_t^\vartheta(x))^\top$ follows from the representation
$$(f(x)\; g\star \mu_t^\vartheta(x))^\top = \big(f(x)\;\E_{\overline{\PP}_\vartheta}[g(X_t)]\big)^\top,$$
the fact that $t \mapsto X_t$ is continuous in probability at $t = 0$ under $\overline{\PP}_\vartheta$ and Lebesgue dominated convergence. Now, let $z=(z_1,z_2)$ with $z_1^2+z_2^2=1$. In order to obtain non-degeneracy, it is sufficient by Proposition \ref{prop: criterion fisher} to show that the function
$$x \mapsto f(x)z_1+g\star\mu_0(x)z_2$$
is non identically zero.  If $f$ does not identically vanishes, we may assume that $z_2 \neq 0$. Then it is sufficient to have that
\begin{equation} \label{eq: g star non deg}
x \mapsto \lambda f(x)+ g\star \mu_0(x)\;\;\text{does not vanish identically for every}\;\;\lambda \neq 0.
\end{equation}
It is then very easy to build families of functions and initial condition $(f,g,\mu_0)$ such that \eqref{eq: g star non deg} is satisfied. For instance, if $\mu_0 = \delta_{x_0}$ for some arbitrary $x_0$, then having $f$ non-identically equal to a constant is sufficient.\\

\subsubsection*{Likelihood equations} Finally, we explicitly solve the likelihood equations. Again, they are of a simple form, and the maximum likelihood estimator $\widehat \vartheta_N^{\,\tt{mle}}$ solves 
\begin{equation} \label{eq: mle equation}
\mathsf A_T^N \widehat \vartheta_N^{\,\tt{mle}} \,= \mathsf B_T^N,
\end{equation}
where $\mathsf A_T^N$ and $\mathsf B_T^N$ are defined via \eqref{eq: integ A} and \eqref{eq: integ B}, with 
$$\mathsf A_t^N(x) =
\left(
\begin{array}{cc}
f(x)^2 & f(x)\langle g(x-\cdot),\mu_t^{(N)}\rangle  \\
 f(x)\langle g(x-\cdot),\mu_t^{(N)}\rangle   &   \langle g(x-\cdot),\mu_t^{(N)}\rangle^2  
 \end{array}
\right),\;\;
\mathsf B_t^N(x) =\left(
\begin{array}{c}
f(x)  \\
 \langle g(x-\cdot),\mu_t^{(N)}\rangle  
 \end{array}
\right).
$$
Again, we have 
$$\mathbb I_{\mathcal G}(\vartheta) =  \left(
\begin{array}{cc}
\int_0^T\langle f^2, \mu_t^\vartheta \rangle dt & \int_0^T \langle f (g \star \mu_t^{\vartheta}),\mu_t^\vartheta \rangle dt \\
 \int_0^T \langle f (g \star \mu_t^{\vartheta}),\mu_t^\vartheta \rangle dt  &   \int_0^T\langle (g \star \mu_t^{\vartheta})^2,\mu_t^\vartheta \rangle dt 
 \end{array}
\right),$$
by taking the limit in $\PP_\vartheta^N$-probability of $\mathsf A_T^N = \int_0^T\langle \mathsf A_t^N(x),\mu_t^{(N)}\rangle dt$. We also know that $\mathsf{det}\, \mathbb I_{\mathcal G}(\vartheta) >0$ by the non-degeneracy property established above thanks to Proposition \ref{prop: criterion fisher}. Hence $\mathsf A_T^N$ is invertible with $\PP_\vartheta^N$-probability that goes to one as $N \rightarrow \infty$ and $\widehat \vartheta_N^{\,\tt{mle}}$ is asymptotically well defined.

\subsection{A double layer potential model}  \label{sec: swarming}
We depart from the structure \eqref{eq: rep exp} of the likelihood as in the two preceding linear-like models and study the model
$$dX_t^i = N^{-1}\sum_{j = 1}^N \nabla U_{\vartheta}(X_t^i-X_t^j)dt+dB_t^i,\;\;1 \leq i \leq N,$$ 
with ambient state space $\R^d$ for $d \geq 1$ and
where $U_{\vartheta}:\R^d \rightarrow \R$ is a family of pairwise potentials of the form
$$U_{\vartheta}(x) = -\vartheta_1\exp(-\vartheta_2|x|^2) +\vartheta_3\exp(-\vartheta_4|x|^2)$$
modelling short range repulsion and long range attraction, where $\vartheta_1, \vartheta_3$ and $\vartheta_2, \vartheta_4$ are respectively the strengths and the lengths of attraction and repulsion. The parameter is $\vartheta = (\vartheta_1,\ldots, \vartheta_4) \in (0,\infty)^4$. As minimal identifiability condition, we impose $\vartheta_2 \neq \vartheta_4$.
Such models are commonly used (in their kinetic version) in swarming modelling, see {\it e.g.} \cite{bolley2011stochastic}.\\

We have $b(\vartheta;t,x,\nu) = \nabla U_\vartheta \star \nu(x)$ and it is readily verified that Assumptions  \ref{reg sigma}, \ref{reg b} and \ref{hyp : reg sigma b} are met. Here, there is no hope to explicitly solve  the likelihood equations, and a numerical scheme has to be implemented. We further investigate the identifiability and non-degeneracy of the model. Note that the assumptions on the drift and the initial condition ensure that for every $\vartheta \in \Theta$, $\mu_t^\vartheta$ is absolutely continuous with a nowhere vanishing density for $t >0$ (we refer to \cite{BoKry} and in particular Corollary 8.2.2).

\subsubsection*{Identifiability}
We study the injectivity of the mapping $\vartheta \mapsto \overline{\PP}_{\vartheta}$ via the injectivity of $\vartheta \mapsto  \big((t,x) \mapsto \nabla U_\vartheta \star \mu_{t}^\vartheta)\big)$, see Assumption \ref{ass: identif product experiment}. If $\vartheta, \vartheta'$ are such that $\overline{\PP}_\vartheta = \overline{\PP}_{\vartheta'}$,  this also implies $\mu_t^{\vartheta} = \mu_t^{\vartheta'}$ for every $t \in [0,T]$.  Hence the condition-
$$ \nabla U_\vartheta \star \mu_{t}^\vartheta =  \nabla U_{\vartheta'} \star \mu_{t}^{\vartheta'}$$
for almost every $t\in [0,T]$ becomes
\begin{equation} \label{eq: fourier crux}
\nabla U_\vartheta \star \mu_{t}^\vartheta =  \nabla U_{\vartheta'} \star \mu_{t}^{\vartheta}
\end{equation}
for almost every $t\in [0,T]$.  Let $\mathcal F(\nu)(\xi) = \int_{\R^d} e^{ix^\top\xi}\,\nu(dx)$ denote a Fourier transform of  $\nu$ (well defined if $\nu$ is a probability measure or an integrable function).  Since $\xi \mapsto \mathcal F(\mu_t)(\xi)$ is continuous and $\mathcal F(\mu_t)(\xi)=1$ (the $\mu_t$ are all probability measures on $\R^d$), there are infinitely many points $\xi \in \R^d$ such that 
$\mathcal F(\mu_t)(\xi) \neq 0$ . Applying $\mathcal F$ on both side of \eqref{eq: fourier crux},  we obtain
$$\mathcal F(\nabla U_\vartheta)(\xi) = \mathcal F(\nabla U_{\vartheta'})(\xi)$$
for such points $\xi$. Moreover
\begin{equation} \label{eq: nabla U}
\nabla U_{\vartheta}(x) = 2\vartheta_1\vartheta_2 x\exp(-\vartheta_2|x|^2)-2\vartheta_3\vartheta_4 x \exp(-\vartheta_4|x|^2),
\end{equation}
and
$$
\mathcal F(\nabla U_\vartheta)(\xi) = i\xi \pi^{d/2}\big(\vartheta_1 \vartheta_2^{-d/2}\exp(-\tfrac{1}{4\vartheta_2}|\xi|^2)-\vartheta_3 \vartheta_4^{-d/2}\exp(-\tfrac{1}{4\vartheta_4}|\xi|^2)\big).
$$
It follows that the condition $\vartheta_2 \neq \vartheta_4$ is sufficient to achieve identifiability, {\it i.e.} $\vartheta = \vartheta'$. Henceforth,  we may parametrise our model via any compact $\Theta \subset (0,\infty)^4$ such that $\vartheta_2 \neq \vartheta_4$.

\subsubsection*{Non-degeneracy} We plan to apply Proposition \ref{prop: criterion fisher}. From \eqref{eq: nabla U}, we have
$$\nabla_\vartheta b(\vartheta; t,x,\nu)^j = \nabla_\vartheta (\nabla U_\vartheta \star \nu(x))^j = G(\vartheta; x)^j\star \nu(x),
$$
with
$$G(\vartheta;x)^j=
\left(
\begin{array}{c}
2\vartheta_2 x_j\exp(-\vartheta_2|x|^2)\\
2\vartheta_1x_j(1-\vartheta_2|x|^2)\exp(-\vartheta_2|x|^2) \\
-2\vartheta_4 x_j\exp(-\vartheta_4|x|^2)\\
-2\vartheta_3x_j(1-\vartheta_4|x|^2)\exp(-\vartheta_4|x|^2)
\end{array}
\right),\;\;j=1,\ldots, d
$$
The mapping 
$$t \mapsto \nabla_\vartheta b([\vartheta, \vartheta']; t,x,\mu_t^\vartheta)^j=\E_{\overline{\PP}_\vartheta}\big[G([\vartheta, \vartheta'];x-X_t)^j\big]$$
is continuous at $t=0$, as a simple consequence of the fact that $t \mapsto X_t$ is continuous in $\overline{\PP}_{\vartheta}$-probability at $t=0$. It is then sufficient to prove that for any $z \in \R^4$ with $z_1^2+z_2^2+z_3^2+z_4^2=1$, one of the functions 
$$x \mapsto (G([\vartheta, \vartheta']; \cdot)^j \star \mu_0(x))^\top z,\;\;j=1,\ldots,d.$$
does not vanish identically. Assume on the contrary that $(G([\vartheta, \vartheta']; \cdot)^j \star \mu_0)^\top z$ is identically $0$ for every $1 \leq j \leq d$. Then this is also the case for $\mathcal F\big((G([\vartheta, \vartheta']; \cdot)^j \star \mu_0)^\top z\big)$. Assume now that $\mathcal F(\mu_0)(\xi) \neq 0$ $d\xi$-a.e. Then from 
$$\mathcal F\big((G([\vartheta, \vartheta']; \cdot)^j \star \mu_0)^\top z\big)(\xi) = \mathcal F\big(G([\vartheta, \vartheta']; \cdot)^j(\xi)\big)^\top z \cdot \mathcal F(\mu_0)(\xi)$$
we must have
$$\xi \mapsto \mathcal F(G([\vartheta, \vartheta']; \cdot)^j)(\xi)^\top z = 0\;\;d\xi-a.e.$$
for every $1 \leq j \leq d$, or equivalently
$$x \mapsto \big(G([\vartheta, \vartheta']; x)^j\big)^\top z = 0\;\;dx-a.e.$$
This is not possible, as proved by
an inspection of the equation
\begin{align*} 
&\int_0^1 \big([\vartheta_2,\vartheta_2']_\lambda x_j\exp(-[\vartheta_2, \vartheta_2']_\lambda|x|^2)z_1+[\vartheta_1, \vartheta_1']_\lambda x_j(1-[\vartheta_2, \vartheta_2']_\lambda|x|^2)\exp(-[\vartheta_2, \vartheta_2']_\lambda|x|^2)z_2\big) d\lambda\\
= &\int_0^1 \big([\vartheta_4,\vartheta_4']_\lambda x_j\exp(-[\vartheta_4, \vartheta_4']_\lambda|x|^2)z_3+[\vartheta_3, \vartheta_3']_\lambda x_j(1-[\vartheta_4, \vartheta_4']_\lambda|x|^2)\exp(-[\vartheta_4, \vartheta_4']_\lambda|x|^2)z_4\big) d\lambda,
\end{align*} 
for almost every $x \in \R^d$,
which further reduces to 
\begin{align} 
&\big([\vartheta_2,\vartheta_2']_{\lambda_x} x_jz_1+[\vartheta_1, \vartheta_1']_{\lambda_x} x_j(1-[\vartheta_2, \vartheta_2']_{\lambda_x}|x|^2)z_2\big)\exp(-[\vartheta_2, \vartheta_2']_{\lambda_x}|x|^2) \nonumber\\
= &\big([\vartheta_4,\vartheta_4']_{\lambda'_x} x_jz_3+[\vartheta_3, \vartheta_3']_{\lambda'_x} x_j(1-[\vartheta_4, \vartheta_4']_{\lambda'_x}|x|^2)z_4\big)\exp(-[\vartheta_4, \vartheta_4']_{\lambda_x'}|x|^2) \label{eq: identif}
\end{align} 
for almost every $x \in \R^d$,
by the mean-value theorem for some $\lambda_x, \lambda_x' \in [0,1]$, that also respectively depend on $(\vartheta_1, \vartheta_2, \vartheta'_1, \vartheta_2',z_1, z_2)$ and $(\vartheta_3, \vartheta_4, \vartheta_3', \vartheta_4', z_3, z_4)$.
A simple sufficient condition is $\vartheta_2 \neq \vartheta_4$: indeed, if $\vartheta_2$ and $\vartheta_4$ take values in disjoint intervals for instance, then for every $\vartheta, \vartheta' \in \Theta$, with $[\vartheta, \vartheta'] \subset \Theta$, we have $[\vartheta_2,\vartheta_2'] \cap [\vartheta_4,\vartheta_4'] = \emptyset$. Then one easily checks that \eqref{eq: identif} cannot hold for sufficiently large $|x|$.\\

If we only need to verify that $\mathbb I_\mathcal G(\vartheta)$ is non-degenerate, it is sufficient to take $\vartheta = \vartheta'$ in \eqref{eq: identif} that simply becomes
\begin{align*} 
\big(\vartheta_2x_jz_1+\vartheta_1x_j(1-\vartheta_2|x|^2)z_2\big)\exp(-\vartheta_2|x|^2) = \big(\vartheta_4 x_jz_3+\vartheta_3x_j(1-\vartheta_4|x|^2)z_4\big)\exp(-\vartheta_4|x|^2)
\end{align*} 
for almost every $x \in \R^d$, in which case having $\vartheta_2 \neq \vartheta_4$ is sufficient (and somewhat easier to obtain then the non-degeneracy of $\mathcal G$ in the sense of Definition \ref{def: non deg}). In conclusion, as soon as $\mathcal F(\mu_0)$ is non-vanishing almost everywhere and $\Theta \subset (0,\infty)^4$ is a compact such that $\vartheta_2 \neq \vartheta_4$, we obtain non-degeneracy of $\mathcal G$.
\subsection{A genuinely non-linear example}

We end-up this section by inspecting an example where the parametrisation $\nu \mapsto b(\vartheta;t,x,\nu)$ is genuinely non-linear  in the measure argument. Consider the model
$$dX_t ^i = F\Big(\vartheta N^{-1}\sum_{j = 1}^Ng(X_t^i-X_t^j)\Big)dt+dB_t^i,\;\;1 \leq i \leq N,$$
with $X_t^i \in \R$ and $\vartheta > 0$ for simplicity.
The functions $F,g:\R \rightarrow \R$ are known and smooth,  $g$ is nonnegative, integrable, with positive mass and $F$ is one-to-one on the positive axis. We have $b(\vartheta;t,x,\nu) = F(\vartheta g\star\nu(x))$.\\

The smoothness of $F$ and $g$ yields Assumptions \ref{reg sigma}, \ref{reg b} and \ref{hyp : reg sigma b}.  As for the identifiability of the model, assume that $\overline{\PP}_{\vartheta} = \overline{\PP}_{\vartheta'}$, and so $\mu_t^{\vartheta} = \mu_t^{\vartheta'}$ as well. If, for almost every $x \in \R$, we have
$$F(\vartheta g \star \mu_t^{\vartheta}(x)) = F(\vartheta' g \star \mu_t^{\vartheta}(x)),$$
then, since $g(x) \geq 0$ and $\mu_t^\vartheta(x) >0$ for every $x \in \R$ and $t>0$, the function $g \star \mu_t^{\vartheta}$ is positive on the whole real line $\R$.  Since $F$ is one-to-one on the positive axis, it follows that 
$$\vartheta g \star \mu_t^{\vartheta}(x) = \vartheta' g \star \mu_t^{\vartheta}(x)$$
and that can only be true if $\vartheta = \vartheta'$ since $g \star \mu_t^{\vartheta}(x)$ is nowhere vanishing. As for the non-degeneracy, we have
\begin{align*}
\partial_\vartheta b(\vartheta; t, x, \mu_t)  & = g\star \mu_t^\vartheta(x) F'\big(\vartheta g\star \mu_t^{\vartheta}(x)\big) = \E_{\overline{\PP}_{\vartheta}}[g(x-X_t)]F'\big(\vartheta\E_{\overline{\PP}_\vartheta}[g(x-X_t)]\big)
\end{align*}
which is continuous at $t=0$ by the continuity in $\overline{\PP}_\vartheta$- probability of $t \mapsto X_t$ at $t=0$. Finally, assume that $\mu_0$ has a positive density. Then
$x \mapsto g\star \mu_0(x) F'\big(\vartheta g\star \mu_0(x)\big)$ is non-identically vanishing since $g\star \mu_0$ is positive and $F'$ is non-vanishing on $(0,\infty)$. We conclude by Proposition \ref{prop: criterion fisher}.

\section{Proof of the main results} \label{section : proofs}

\subsection{Preliminaries: couplings} \label{sec: couplings}
For technical reasons, we will need certain couplings on the canonical space $(\mathcal C^N,\mathcal F^N)$. We now fix $\vartheta \in \Theta$, and introduce, for every $\vartheta' \in \Theta$, the following two processes 
$$X^{(N), \vartheta'} = (X_t^{1, \vartheta'}, \ldots, X_t^{N, \vartheta'})_{t \in [0,T]}$$
and
$$\overline{X}^{(N), \vartheta'} = (\overline{X}_t^{1, \vartheta'}, \ldots, \overline{X}_t^{N, \vartheta'})_{t \in [0,T]}$$
defined on $(\mathcal C^N, \mathcal F^N)$ by
\begin{equation} \label{eq: def coupling}
X_t^{i,\vartheta'} = X_0^i+\int_0^t b(\vartheta';s,X_s^{i,\vartheta'}, \mu_t^{(N), \vartheta'})ds+\int_0^t \sigma(s,X_s^{i,\vartheta'})dB^{i,N,\vartheta}_s,
\end{equation}
with $\mu_t^{(N),\vartheta'} = N^{-1}\sum_{i =1}^N \delta_{X_t^{i,\vartheta'}}$ and
\begin{equation} \label{eq: def coupling product}
\overline{X}_t^{i,\vartheta'} = X_0^i+\int_0^t b(\vartheta';s,\overline{X}_s^{i,\vartheta'}, \mu_t^{\vartheta'})ds+\int_0^t \sigma(s,\overline{X}_s^{i,\vartheta'})dB^{i,N,\vartheta}_s,
\end{equation}
and 
$$B^{i,N,\vartheta}_t = \int_0^t c^{-1/2}(s,X_s^i)(dX_s^i-b(\vartheta, s,X_s^i, \mu_s^{(N)})ds).$$ 
Note that $X^{(N), \vartheta'}$ and $\overline{X}^{(N), \vartheta'}$ actually depend on $\vartheta$ pathwise via $(B_t^{i,N,\vartheta})_{t \in [0,1]}$ under $\mathbb P_{\vartheta}^N$ but not their laws! Indeed, the $(B_t^{i,N,\vartheta})_{t \in [0,1]}$ are standard Brownian motions under $\mathbb P_{\vartheta}^N$. For notational simplicity, we omit the dependence upon $\vartheta$ here. Otherwise, we write $X_t^{i,\vartheta',\vartheta}$ or $\overline{X}_t^{i,\vartheta',\vartheta}$. Thus $X^{(N), \vartheta'}$ and $\overline{X}^{(N), \vartheta'}$ are defined as functions of $X^{(N)}$ (as strong solutions of \eqref{eq: def coupling} and \eqref{eq: def coupling product}) and have law $\PP_{\vartheta'}^N$ and $\overline{\PP}_{\vartheta'}^{\otimes N}$ under $\mathbb P_{\vartheta}^N$. This is a convenient way to couple $\PP_{\vartheta'}^N$ and $\overline{\PP}_{\vartheta'}^{\otimes N}$ while still working with the canonical process under $\mathbb P_{\vartheta}^N$. We write $\mu_t^{(N)} = N^{-1}\sum_{i  = 1}^N \delta_{X_t^i}$, for the empirical measure of the canonical process. We also introduce
$$\overline{\mu}_t^{(N),\vartheta'} = N^{-1}\sum_{i  = 1}^N \delta_{\overline{X}_t^{i,\vartheta'}}$$
and write $\mu_t^{(N), \vartheta',\vartheta}$ and  $\overline{\mu}_t^{(N),\vartheta', \vartheta}$ whenever we want to emphasise that the coupling is constructed with the processes $(B^{i,N,\vartheta}_t)_{t \in [0,T]}$ for $1 \leq i \leq N$.
We have the following approximation results:
\begin{lem} \label{lem: couplings}
For every $\vartheta,\vartheta' \in \Theta$ and every $r\geq 1$, we have
\begin{equation} \label{eq: coupling 1}
   \sup_{t \in [0,T]}\EE_{\PP_\vartheta^N}\big[\mathcal{W}_1\big(\mu^{(N),\vartheta', \vartheta}_t,\mu^{(N)}_t\big)^r\big] \leq  \sup_{t \in [0,T]} \EE_{\PP_{\vartheta}^N}\big[  N^{-1}\sum_{i=1}^N | X^{i,\vartheta', \vartheta}_t  - X^{i}_t |^{r} \big] \leq C |\vartheta'-\vartheta|^r,   
 \end{equation}
\begin{equation}  \label{eq: coupling 2}
     \sup_{t \in [0,T]}\EE_{\PP_{\vartheta}^N}\big[\mathcal W_1\big(\bar{\mu}^{(N),\vartheta', \vartheta}_t,\overline{\mu}^{(N),\vartheta, \vartheta}_t\big)^r\big]  \leq  \sup_{t \in [0,T]} \EE_{\PP_{\vartheta}^N}\big[ N^{-1}\sum_{i=1}^N | \overline{X}^{i,\vartheta',\vartheta}_t  - \overline{X}^{i,\vartheta, \vartheta}_t |^{r} \big] \leq C |\vartheta'-\vartheta|^r.   
 \end{equation}
 There exists $\delta >0$ such that for every $r \geq 1$:
 \begin{align}
   &\sup_{t \in [0,T], \vartheta \in \Theta}  \EE_{\PP_{\vartheta}^N}\big[\mathcal W_1\big(\mu^{(N)}_t ,\mu^{\vartheta}_t \big)^{r}\big]  \leq CN^{-\delta r}, \label{eq: first last coupling}\\
   &\sup_{t \in [0,T], \vartheta \in \Theta}  N^{-1}\sum_{i=1}^N \EE_{\PP_{\vartheta}^N}[ | X^{i}_t  - \overline{X}^{i,\vartheta, \vartheta}_t |^{r} ] \leq CN^{-\delta r}  \label{eq: last coupling}
\end{align}
as $N \rightarrow \infty$.
\end{lem}
The proof is given in Appendix \ref{app: lem couplings}.

\subsection{Proof of Theorem \ref{thm: LAN prop}}

We prove a slightly stronger result, namely a uniform type LAN condition, following Chapter III of \cite{ibragimov2013statistical}. Let  $(u_N)_{N \geq 1}$ be a sequence of $\R^p$ such that $u_N \rightarrow u$ and $(\vartheta_N)_{N \geq 1}$ a sequence of $\Theta$ such that $\vartheta_N+(N\mathbb I_\mathcal G(\vartheta_N))^{-1/2} u_N \in \Theta$ for large enough $N$ and such that $\vartheta_N \rightarrow \vartheta$, for some $\vartheta$ such that $\mathbb I_{\mathcal G}(\vartheta) >0$. We claim that
\begin{equation} \label{eq: LAN unif}
\zeta_N(\vartheta_N;u_N) = \log \frac{d\PP_{\vartheta_N+(N\mathbb I_\mathcal G(\vartheta_N))^{-1/2} u_N}^N}{d\PP_{\vartheta_N}^N} = u^\top \Gamma_N-\tfrac{1}{2}|u|^2+r_N(\vartheta_N, u_N),
\end{equation}
where $\Gamma_N \rightarrow \mathcal N(0,\mathrm{Id}_{\R^p})$ in distribution under $\PP_\vartheta^N$ and $r_N(\vartheta_N, u_N) \rightarrow 0$ in $\PP_\vartheta^N$-probability. Clearly, \eqref{eq: LAN unif} implies \eqref{eq: LAN} and thus Theorem \ref{thm: LAN prop}. Note that since $\mathbb I_{\mathcal G}(\vartheta) >0$ we have 
that $\mathbb I_\mathcal G(\vartheta_N)$ is invertible for large enough $N$, thanks to the continuity of the mapping $\vartheta \mapsto \mathbb I_{\mathcal G}(\vartheta)$, recall Proposition \ref{prop: fisher cont}. The  asymptotic expansion \eqref{eq: LAN unif} is therefore meaningful for large enough $N$.\\

We further write $\mathbb I(\vartheta)$ for $\mathbb I_{\mathcal G}(\vartheta)$.\\ 

\noindent {\it Step 1.} {\it (Preliminary expansion.)} We have
\begin{align*}
&\zeta_N(\vartheta_N;u_N)  \\
& = \sum_{i=1}^{N}  \int_0^T \big( (c^{-1/2}b)(\vartheta_N+ (N {\mathbb I(\vartheta_N)})^{-1/2} u_N; t, X^{i}_t,\mu^{N}_t)-(c^{-1/2}b)(\vartheta_N; t, X^{i}_t,\mu^{N}_t) \big)^{\top} dB^{i, N,\vartheta_N}_t\\
& -\frac{1}{2} \sum_{i=1}^{N} \int_0^T \big|(c^{-1/2}b)(\vartheta_N + (N {\mathbb I(\vartheta_N)})^{-1/2} u_N; t, X^{i}_t,\mu^{N}_t) - (c^{-1/2}b)(\vartheta_N; t, X^{i}_t,\mu^{N}_t)\big|^2 dt, 
\end{align*}
where the $B^{i, N,\vartheta_N}_t = \int_0^t c^{-1/2}(s,X_s^i)(dX_s^i-b(\vartheta_N;s,X_s^i,\mu_s^{(N)})ds)$, $1 \leq i \leq N$ are independent Brownian motions on $\R^d$ under $\mathbb \PP_{\vartheta_N}^N$. A first-order Taylor's expansion therefore yields the representation,
\begin{align*}
\zeta_N(\vartheta_N;u_N) 
& = u_N^\top (\mathbb I(\vartheta_N)^{-1/2})^\top \Delta_{N,\vartheta_N}(u_N)-\tfrac{1}{2}u_N^\top (\mathbb I(\vartheta_N)^{-1/2})^\top \,\widetilde{\mathbb I}_{N, \vartheta_N}(u_N)\mathbb I(\vartheta_N)^{-1/2} u_N 
\end{align*}
where
\begin{align*}
 \Delta_{N,\vartheta_N}(u) = N^{-1/2}\sum_{i=1}^N\sum_{j = 1}^d\int_0^T \nabla_\vartheta(c^{-1/2}b)^j([\vartheta_N, \vartheta_N+(N \mathbb I(\vartheta_N))^{-1/2}u];t,X_t^i,\mu_t^N )d(B^{i, N, \vartheta_N}_t)^j
\end{align*}
and
\begin{align*}
 \widetilde{\mathbb I}_{N, \vartheta_N}(u) & = N^{-1} \sum_{i = 1}^N\sum_{j=1}^d \int_0^T \nabla_\vartheta(c^{-1/2}b)^j([\vartheta_N,\vartheta_N+(N \mathbb I(\vartheta_N))^{-1/2}u];t,X_t^i,\mu_t^N ) \\
 &\;\;\;\;\times \big(\nabla_\vartheta(c^{-1/2}b)^j([\vartheta_N, \vartheta_N+(N \mathbb I(\vartheta_N))^{-1/2}u];t,X_t^i,\mu_t^N) \big)^\top dt,
\end{align*}
with the notation $\phi([\vartheta, \vartheta']) = \int_0^1\phi(\vartheta+\lambda(\vartheta'-\vartheta))d\lambda$ that we introduced before.
We rewrite the above expansion as
\begin{align*}
\zeta_N(\vartheta_N;u_N) & = u_N^\top (\mathbb I(\vartheta_N)^{-1/2})^\top \Delta_{N,\vartheta_N}(0) - \tfrac{1}{2}|u|^2\\
&+u_N^\top (\mathbb I(\vartheta_N)^{-1/2})^\top \Delta_{N,\vartheta_N}(u_N)-u^\top  (\mathbb I(\vartheta_N)^{-1/2})^\top \Delta_{N,\vartheta_N}(0)\\
&-\tfrac{1}{2}\big(u_N^\top (\mathbb I(\vartheta_N)^{-1/2})^\top \,\widetilde{\mathbb I}_{N, \vartheta_N}(u_N)\mathbb I(\vartheta_N)^{-1/2} u_N -|u|^2\big)
\end{align*}
and thus \eqref{eq: LAN unif} follows from
\begin{equation} \label{eq: TCL}
(\mathbb I(\vartheta_N)^{-1/2})^\top \Delta_{N,\vartheta_N}(0)  \rightarrow \mathcal N(0,\mathrm{Id}_{\R^p})
\end{equation} 
under $\PP_{\vartheta_N}^N$ in distribution together with the convergence to $0$ of the last two components.\\

\noindent {\it Step 2.} {\it (Convergence of the Gaussian part.)} We prove \eqref{eq: TCL} or equivalently, the convergence
$$\xi^\top (\mathbb I(\vartheta_N)^{-1/2})^\top \Delta_{N,\vartheta_N}(0) = \sum_{q,q'=1}^d \xi_q (\mathbb I(\vartheta_N)^{-1/2})_{q'q}  \Delta_{N,\vartheta_N}(0)_{q'}  \rightarrow \mathcal N(0,|\xi|^2)$$
in distribution under $\PP_{\vartheta_N}^N$ for every $\xi \in \R^p$. We apply a classical semimartingale convergence result, following for instance Jacod and Shiryaev \cite{JS} (Corollary 3.24). For $t \in [0,T]$, the process 
$$\Delta_{N,\vartheta_N}(0)_t  =  N^{-1/2}\sum_{i=1}^N\sum_{j = 1}^d\int_0^t \nabla_\vartheta(c^{-1/2}b)^j(\vartheta_N;s,X_s^i,\mu_s^N )d\lambda \big)d(B^{i, N, \vartheta_N}_s)^j
$$
is a continuous local martingale under $\mathbb P_{\vartheta_N}^N$ and so is  $\big(\xi^\top (\mathbb I(\vartheta_N)^{-1/2})^\top \Delta_{N,\vartheta_N}(0)_t\big)_{t \in [0,T]}$. It coincides with $\xi^\top (\mathbb I(\vartheta_N)^{-1/2})^\top \Delta_{N,\vartheta_N}(0)$ at $t=T$ and has predictable compensator
\begin{align*}
 \big\langle \xi^\top (\mathbb I(\vartheta_N)^{-1/2})^\top &\Delta_{N,\vartheta_N}(0)_\cdot \big\rangle_t 
= \sum_{q_k,q_k'=1}^p \xi_{q_1}\xi_{q_2} (\mathbb I(\vartheta_N)^{-1/2})_{q_2'q_2}(\mathbb I(\vartheta_N)^{-1/2})_{q_1'q_1} \times \\
&N^{-1}\sum_{i = 1}^N\sum_{j =1}^d \int_0^t \partial_{\vartheta_{q'_1}}(c^{-1/2}b)^j(\vartheta_N; s,X_s^i, \mu_s^{(N)})\partial_{\vartheta_{q'_2}}(c^{-1/2}b)^j(\vartheta_N; s,X_s^i, \mu_s^{(N)})ds,
\end{align*}
that converges to 
\begin{align*}
&\sum_{q_k,q_k'=1}^p \xi_{q_1}\xi_{q_2} (\mathbb I(\vartheta)^{-1/2})_{q_2'q_2}(\mathbb I(\vartheta)^{-1/2})_{q_1'q_1} \times \\
&\sum_{j =1}^d \int_0^t \E_{\overline{\PP}_\vartheta}\big[\partial_{\vartheta_{q'_1}}(c^{-1/2}b)^j(\vartheta; s,X_s^i, \mu_s^{\vartheta})\partial_{\vartheta_{q'_2}}(c^{-1/2}b)^j(\vartheta; s,X_s^i, \mu_s^\vartheta)\big]ds,
\end{align*}
in $\PP_{\vartheta_N}^N$-probability, but that last quantity is exactly $|\xi|^2$ at $t=T$, which proves \eqref{eq: TCL}. As for the last convergence in probability, it is a simple consequence of the continuity of $\vartheta \mapsto \mathbb I(\vartheta)$, see Proposition \ref{prop: fisher cont} and the following lemma
\begin{lem} \label{lem: conv prob}
Let $\beta>0$ and $\phi: \Theta \times [0,T] \times \R^d \times \mathcal P_\beta \rightarrow \R$ be such that for some $C,\alpha >0$, we have
$$\sup_{t \in [0,T]}|\phi(\vartheta';t,x',\nu')-\phi(\vartheta;t,x,\nu)| \leq C(|\vartheta'-\vartheta|+|x'-x|+\mathcal W_1(\nu',\nu))(1+|x|^\alpha+|x'|^\alpha+\mathfrak m_\beta(\nu)+\mathfrak m_\beta(\nu')).$$
Then, there exists $0 < \delta \leq 1/2$ such that for every $t \in [0,T]$ and every $m>0$, we have
$$\E_{\PP_{\vartheta_N}^N}\Big[\Big|N^{-1}\sum_{i = 1}^N\int_0^t \phi(\vartheta_N; s, X_s^i, \mu_s^{(N)})ds - \int_0^t \int_{\R^d}\phi(\vartheta; s,x,\mu_s^\vartheta)\mu_s^\vartheta(dx)ds\Big|^m\Big] \leq C(|\vartheta_N-\vartheta|^m+N^{-\delta m}).$$
\end{lem} 
We apply Lemma \ref{lem: conv prob} to  $\phi(\vartheta;s,x,\nu)= \partial_{\vartheta_{q_1'}} (c^{-1/2}b)^j(\vartheta; s,x,\nu  )  \partial_{\vartheta_{q_2'}} (c^{-1/2}b)^j(\vartheta; s, x,\nu )$, thanks to Assumption \ref{hyp : reg sigma b}. The proof is given in Appendix \ref{app: LGN general}.\\

\noindent {\it Step 3.} {\it (Convergence of the remainder terms.)} We first prove
\begin{equation} \label{eq: first remainder}
u_N^\top (\mathbb I(\vartheta_N)^{-1/2})^\top \Delta_{N,\vartheta_N}(u_N)-u^\top  (\mathbb I(\vartheta_N)^{-1/2})^\top \Delta_{N,\vartheta_N}(0) \rightarrow 0
\end{equation}
in $\PP_{\vartheta_N}^N$-probability. Since $\mathbb I(\vartheta_N)^{-1/2}$ is well defined for large enough $N$ and converges to $\mathbb I(\vartheta)^{-1/2}$ and $u_N \rightarrow u$, it is sufficient to prove $\Delta_{N,\vartheta_N}(u_N)- \Delta_{N,\vartheta_N}(0) \rightarrow 0$ in $\PP_{\vartheta_N}^N$-probability. Introduce the process
$$G_{N}^{i, r}(\vartheta, u)_{t} = \int_0^1\big(\partial_{\vartheta_r}(c^{-1/2}b)(\vartheta+\lambda (N\mathbb I(\vartheta))^{-1/2}u;t,X_t^i,\mu_t^{(N)})-\partial_{\vartheta_r}(c^{-1/2}b)(\vartheta;t,X_t^i,\mu_t^{(N)})\big)d\lambda$$
for $1 \leq r \leq p$ and $t \in [0,T]$.  By It\^o's isometry
\begin{align*}
\E_{\PP_{\vartheta_N}^N}\big[|\Delta_{N,\vartheta_N}(u_N)- \Delta_{N,\vartheta_N}(0)|^2\big] & = \sum_{r = 1}^p\E_{\PP_{\vartheta_N}^N}\big[|N^{-1/2}\sum_{i = 1}^N \int_0^TG_{N}^{i, r}(\vartheta_N, u_N)_{t}^\top dB_t^{i,N,\vartheta_N}\big|^2 \big] \\
& = \sum_{r=1}^p N^{-1}\sum_{i = 1}^N \int_0^T\E_{\PP_{\vartheta_N}^N}\big[|G_{N}^{i, r}(\vartheta_N, u_N)_{t}|^2\big]dt.
\end{align*}
Moreover
\begin{align*}
|G_{N}^{i, r}(\vartheta_N, u_N)_{t}|^2  & \leq \sum_{j = 1}^d \sup_{\vartheta \in \Theta}\big|(\nabla_\vartheta \partial_{\vartheta_r}c^{-1/2}b)^j(\vartheta;t,X_t^i,\mu_t^{(N)}) \lambda (N\mathbb I(\vartheta_N))^{-1/2}u_N\big|^2 \\
& \leq CN^{-1}(1+|X_t^i|^{2r_1}+{\mathfrak m}_{r_2}(\mu_t^{(N)})^2),
\end{align*}
for large enough $N$, thanks to Assumption \ref{hyp : reg sigma b}. We conclude
$$\sup_{t \in [0,T]}\E_{\PP_{\vartheta_N}^N}\big[|G_{N}^{i, r}(\vartheta_N, u_N)_{t}|^2\big] \leq CN^{-1}$$
for large enough $N$ by Lemma \ref{lem: moment bound} and \eqref{eq: first remainder} follows.\\

The convergence of the second remainder term is a simple consequence of 
$\widetilde{\mathbb I}_{N, \vartheta_N}(u_N) \rightarrow \mathbb I(\vartheta)$
thanks to Lemma \ref{lem: conv prob} together with the continuity of $\vartheta \mapsto \mathbb  I(\vartheta)$ and Proposition \ref{prop: fisher cont}. The Proof of Theorem \ref{thm: LAN prop} is complete for the experiment $\mathcal E^N$.\\

\noindent {\it Step 4. (The case of the experiment $\mathcal G^{\otimes N}$.)} We now easily extend the previous results to the  experiment $\mathcal G^{\otimes N}$. Since it is a product of the same experiment $\mathcal G$, it is tempting to use classical criterions for IID data. However, from a simple glance at the structure of the previous computations, it suffices to retrace Steps 1 to 3 replacing $\mu_t^{(N)}$ by $\mu_t^{\vartheta_N}$ and $\PP_{\vartheta_N}^N$ by $\overline{\PP}_{\vartheta_N}^{\otimes N}$ that actually turn out to be simpler. We omit the details.


\subsection{Proof of Theorem \ref{thm: MLE prop}} \label{sec: proof of theorem mle}
We plan to apply the classical theory of Ibragimov-Hasminski,  and more specifically Theorem III.1.1 of \cite{ibragimov2013statistical}.We introduce the notation
$$\mathcal Z_N(\vartheta;u) = \frac{d\PP_{\vartheta+(N\mathbb I_\mathcal G(\vartheta))^{-1/2} u}^N}{d\PP_{\vartheta}^N}.$$
We first establish two key regularity properties of the likelihood process.\\

\noindent {\it Step 1.} (A regularity property for the likelihood process.) Here, we prove that for any $r \geq 2$
\begin{equation} \label{eq: smooth to be proved}
\E_{\PP_\vartheta^N}\big[\big|\mathcal Z_N(\vartheta;u)^{1/r}-\mathcal Z_N(\vartheta;v)^{1/r}\big|^r\big] \leq C(1+\kappa^\gamma) |u-v|^{\gamma'},
\end{equation}
for some positive $\gamma,\gamma'$, uniformly in $u,v$ such that $\mathcal Z_N(\vartheta;u)$ and $\mathcal Z_N(\vartheta;v)$ are well-defined and $|u|, |v|$ are bounded by $\kappa >0$. Pick any $r \geq 2$. By a first-order expansion
\begin{align}
\E_{\PP_\vartheta^N}\big[\big|\mathcal Z_N(\vartheta;u)^{1/r}-\mathcal Z_N(\vartheta;v)^{1/r}\big|^r\big]& \leq |u-v|^r\E_{\PP_\vartheta^N}\big[\big|\int_0^1\mathcal \nabla_u (\mathcal Z_N)^{1/r}(\vartheta;u+\lambda(v-u)) d\lambda\big|^r\big]  \nonumber\\
& \leq C|u-v|^r \int_0^1 \sum_{q = 1}^p\E_{\PP_\vartheta^N}\big[\big|\partial_{u_q}((\mathcal Z_N)^{1/r})(\vartheta;u+\lambda(v-u))\big|^r\big]d\lambda. \label{eq: decomp modulus}
\end{align}
Define, for $t \in [0,T]$, the random process 
$$\phi_t(\vartheta;X^i,\mu^{(N)}) = \exp\Big(\int_0^t (c^{-1}b)(\vartheta;s,X_s^i, \mu_s^{(N)})dX_s^i-\tfrac{1}{2}\int_0^t |(c^{-1/2}b)(\vartheta;s,X_s^i, \mu_s^{(N)})|^2ds\Big).$$
Since
$$\mathcal Z_N(\vartheta;u) = \prod_{i = 1}^N\frac{\phi_T^{1/r}(\vartheta  +(N\mathbb I_{\mathcal G}(\vartheta))^{-1/2}u, X^i, \mu^{(N)})}{\phi_T^{1/r}(\vartheta, X^i, \mu^{(N)})},$$
we have
\begin{align*}
\partial_{u_q}((\mathcal Z_N)^{1/r})(\vartheta; u) & = \sum_{i = 1}^N \frac{\partial_{u_q}(\phi_T^{1/r}(\vartheta+(N\mathbb I_{\mathcal G}(\vartheta))^{-1/2}u;X^i,\mu^{(N)}))}{\phi_T^{1/r}(\vartheta; X^i,\mu^{(N)})} \\
&\;\;\;\;\times \prod_{i'\neq i}  \frac{\phi_T^{1/r}(\vartheta+(N\mathbb I_{\mathcal G}(\vartheta))^{-1/2}u;X^{i'},\mu^{(N)})}{\phi_T^{1/r}(\vartheta; X^{i'},\mu^{(N)})}\\
& = \sum_{i = 1}^N \frac{\nabla_\vartheta(\phi_T^{1/r})(\vartheta+(N\mathbb I_{\mathcal G}(\vartheta))^{-1/2}u;X^i,\mu^{(N)})^\top \big((N\mathbb I_{\mathcal G}(\vartheta))^{-1/2}\big)_{\cdot q}}{\phi_T^{1/r}(\vartheta+(N\mathbb I_{\mathcal G}(\vartheta))^{-1/2}u; X^i,\mu^{(N)})} \\
&\;\;\;\;\times \prod_{i'=1}^N  \frac{\phi_T^{1/r}(\vartheta+(N\mathbb I_{\mathcal G}(\vartheta))^{-1/2}u;X^{i'},\mu^{(N)})}{\phi_T^{1/r}(\vartheta; X^{i'},\mu^{(N)})}\\
&=(\mathcal Z_N)^{1/r}(\vartheta; u)\big((N\mathbb I_{\mathcal G}(\vartheta))^{-1/2}\big)_{\cdot q} \times \\
&\;\;\;\;\;\sum_{i = 1}^N \nabla_\vartheta \big( \log(\phi_T^{1/r})\big)(\vartheta +(N\mathbb I_{\mathcal G}(\vartheta))^{-1/2}u, X^i, \mu^{(N)}).
\end{align*}
Interpreting $\mathcal Z_N(\vartheta; u)$ as a Radon-Nikodym derivative entails
\begin{align}
&\E_{\PP_\vartheta^N}\big[\big|\partial_{u_q}((\mathcal Z_N)^{1/r})(\vartheta;u)\big|^r\big] \nonumber \\
& =|\big(\mathbb I_{\mathcal G}(\vartheta))^{-1/2}\big)_{\cdot q}|^r \E_{\PP_{\vartheta+(N\mathbb I_{\mathcal G}(\vartheta))^{-1/2}u}^N}\big[\big|N^{-1/2}\sum_{i = 1}^N \nabla_\vartheta \big( \log(\phi_T^{1/r})\big)(\vartheta +(N\mathbb I_{\mathcal G}(\vartheta))^{-1/2}u)\big|^r\big] \label{eq: rep change of prob smooth}
\end{align}
by a change of probability between $\PP_{\vartheta}^N$ and $\PP_{\vartheta+(N\mathbb I_{\mathcal G}(\vartheta))^{-1/2}u}^N$.
Next, by definition of $\phi_t$, we have, for every $\vartheta' \in \Theta$
\begin{align*}
\partial_{\vartheta_{q'}}(\log (\phi_T^{1/r}))(\vartheta'; X^i ,\mu^{(N)}) & = \frac{1}{r}\int_0^T \partial_{\vartheta_{q'}}(c^{-1}b)(\vartheta'; t;X_t^i,\mu_t^{(N)})dX_t ^i \\
&\;\;\;-\frac{1}{2r}\int_0^T 2\partial_{\vartheta_{q'}}(c^{-1/2}b)(\vartheta'; t, X_t ^i,\mu_t^{(N)})^\top (c^{-1/2}b)(\vartheta'; t, X_t ^i,\mu_t^{(N)})dt\\
& = \frac{1}{r}\int_0^T \partial_{\vartheta_{q'}}(c^{-1/2}b)(\vartheta'; t;X_t^i,\mu_t^{(N)})dB_t^{i,N,\vartheta'}, 
\end{align*}
where the $(B^{i, N, \vartheta'}_t)_{t \in [0,T]} = (\int_0^t c^{-1/2}(s,X_s^i)(dX_s^i-b(\vartheta'; s,X_s^i,\mu_s^{(N)})ds)_{t \in [0,T]}$, $1 \leq i \leq N$ are independent Brownian motions on $\R^d$ under $\PP_{\vartheta'}^N$. Plugging-in this representation in \eqref{eq: rep change of prob smooth} at $\vartheta'=\vartheta +(N\mathbb I_{\mathcal G}(\vartheta))^{-1/2}u$, we infer
\begin{align*}
&\E_{\PP_\vartheta^N}\big[\big|\partial_{u_q}((\mathcal Z_N)^{1/r})(\vartheta;u)\big|^r\big] \\
& \leq C \E_{\PP_{\vartheta'}^N}\big[\big|N^{-1/2}\sum_{i = 1}^N \int_0^T \partial_{\vartheta_{q'}}(c^{-1/2}b)(\vartheta'; t;X_t^i,\mu_t^{(N)})dB_t^{i,N,\vartheta'}\big|^r\big] \\
& \leq C \E_{\PP_{\vartheta'}^N}\big[\big|N^{-1}\sum_{i = 1}^N \int_0^T |\partial_{\vartheta_{q'}}(c^{-1/2}b)(\vartheta'; t;X_t^i,\mu_t^{(N)})|^2dt\big|^{r/2}\big] \\
& \leq C N^{-1}\sum_{i = 1}^N \int_0^T(1+\E_{\PP_{\vartheta'}^N}\big[|X_t^i|^{r_1r}\big]+\mathfrak m_{r_2}(\mu_t^{(N)})^r) dt \\
& \leq C(1+\int_0^T\E_{\PP_{\vartheta'}^N}\big[|X_t^i|^{r\max(r_1,r_2)}\big]dt).
\end{align*} 
where we succesively used the Burckholder-Davis-Gundy, Assumption \ref{hyp : reg sigma b} and the fact that $r \geq 2$. Now, we claim that with $\vartheta' = \vartheta+(N\mathbb I_{\mathcal G}(\vartheta))^{-1/2}(u+\lambda (v-u))$, we have
\begin{equation} \label{eq: claim weird}
\int_0^T\E_{\PP_{\vartheta+(N\mathbb I_{\mathcal G}(\vartheta))^{-1/2}(u+\lambda (v-u))}^N}\big[|X_t^i|^{r\max(r_1,r_2)}\big]dt \leq C(1+\kappa^\gamma)
\end{equation}
for some $\gamma >0$, uniformly in $|u|, |v|$ bounded by $\kappa$ and where $C$ depends on $\Theta$ only. Indeed, keeping up with the abbreviation $\vartheta'$, we have
$$\E_{\PP_{\vartheta'}}\big[|X_t^i|^{r\max(r_1,r_2)}\big] \leq C\big(\E_{\PP_{\vartheta'}}\big[|X_t^i-X_t^{i,\vartheta}|^{r\max(r_1,r_2)}\big]+\E_{\PP_{\vartheta'}}\big[|X_t^{i,\vartheta}|^{r\max(r_1,r_2)}\big].$$
By \eqref{eq: coupling 1} of Lemma \ref{lem: couplings}
\begin{align*}
\E_{\PP_{\vartheta'}}\big[|X_t^i-X_t^{i,\vartheta}|^{r\max(r_1,r_2)}\big] & \leq C|\vartheta-\vartheta'|^{r\max(r_1,r_2)} \\
& = C|(N\mathbb I_{\mathcal G}(\vartheta))^{-1/2}(u+\lambda (v-u))|^{r\max(r_1,r_2)} \\
& \leq CN^{-r\max(r_1,r_2)/2}\kappa^{r\max(r_1,r_2)}
\end{align*}
and the second term that only depends on $\vartheta$ by coupling is uniformly bounded by Lemma \ref{lem: moment bound}. The estimate \eqref{eq: claim weird} follows.
Going back to \eqref{eq: decomp modulus}, we conclude
\begin{align*}
\E_{\PP_\vartheta^N}\big[\big|\mathcal Z_N(\vartheta;u)^{1/r}-\mathcal Z_N(\vartheta;u)^{1/r}\big|^r\big]& \leq C|u-v|^r(1+\kappa^{r\max(r_1,r_2)})
\end{align*}
and \eqref{eq: smooth to be proved} is established with $\gamma = r\max(r_1,r_2)$ and $\gamma'=r$.\\

\noindent {\it Step 2.} Here we prove a moment bound for the likelihood ratio process, namely, for every $r >0$
\begin{equation} \label{eq: moment bound likelihood}
\E_{\PP_\vartheta^N}\big[\mathcal Z_N(\vartheta;u)^{1/2}\big] \leq C|u|^{-r},
\end{equation} 
uniformly in $\vartheta \in \Theta$ and $u = (N\mathbb I_{\mathcal G}(\vartheta))^{1/2}(\vartheta'-\vartheta)$ with $\vartheta'\in \Theta$. Introducing for $t\in [0,T]$ the $\mathbb P_{\vartheta}^N$-martingale
$$\mathcal M_t^N(\vartheta;u) = \sum_{i = 1}^N\int_0^t \big((c^{-1/2}b)(\vartheta+(N\mathbb I_{\mathcal G}(\vartheta))^{-1/2}u;s,X_s^i;\mu_s^{(N)})-(c^{-1/2}b)(\vartheta;s,X_s^i,\mu_s^{(N)})\big)^\top dB_s^{i,N,\vartheta},$$
we have
$$\mathcal Z_N(\vartheta;u) = \exp\big(\mathcal M_T^N(\vartheta;u)-\tfrac{1}{2}\langle \mathcal M_\cdot^N(\vartheta; u)\rangle_T\big).$$
It follows that
\begin{align*}
\E_{\PP_\vartheta^N}\big[\mathcal Z_N(\vartheta;u)^{1/2}\big] &= \E_{\PP_\vartheta^N}\Big[\exp(-\tfrac{1}{2}\mathcal M_T^N(\vartheta;u)-\tfrac{3}{16}\langle \mathcal M_\cdot^N(\vartheta; u)\rangle_T) \exp\big(-\tfrac{1}{16}\langle \mathcal M_\cdot^N(\vartheta; u)\rangle_T\big)\Big] \\
& \leq \E_{\PP_\vartheta^N}\big[\exp(-\tfrac{3}{4} \mathcal M_T^N(\vartheta;u)-\tfrac{1}{2}\langle \tfrac{3}{4}\mathcal M_\cdot^N(\vartheta; u)\rangle_T)\big]^{\frac{2}{3}} \E_{\PP_\vartheta^N}\big[\exp\big(-\tfrac{3}{16}\langle \mathcal M_\cdot^N(\vartheta; u)\rangle_T\big)\big]^{\frac{1}{3}} \\
& \leq \E_{\PP_\vartheta^N}\big[\exp\big(-\tfrac{3}{16}\langle \mathcal M_\cdot^N(\vartheta; u)\rangle_T\big)\big]^{1/3},  
\end{align*}
thanks to H\"older's inequality and the martingale property of $(\tfrac{3}{4}\mathcal M_t^N(\vartheta;u))_{t \in [0,T]}$. With the help of the parametrisation $u = (N\mathbb I_{\mathcal G}(\vartheta))^{1/2}(\vartheta'-\vartheta)$, we rewrite $\langle \mathcal M_\cdot^N(\vartheta; u)\rangle_T $ as
\begin{align*}
 & \sum_{i = 1}^N\int_0^T \big|(c^{-1/2}b)(\vartheta';t,X_t^i,\mu_t^{(N)})-(c^{-1/2}b)(\vartheta;t;X_t^i,\mu_t^{(N)})\big|^2 dt \\
&= \sum_{i =1}^N\sum_{j = 1}^d \int_0^T \big(((N\mathbb I_{\mathcal G}(\vartheta))^{-1/2}u)^\top\int_0^1 \nabla_{\vartheta}(c^{-1/2}b)^j(\vartheta+\lambda(\vartheta'-\vartheta) ;t,X_t^i,\mu_t^{(N)})d\lambda \big)^2dt \\
&= u^\top \mathbb I_{\mathcal G}(\vartheta)^{-1/2}\Sigma^N([\vartheta, \vartheta'];u) \mathbb I_{\mathcal G}(\vartheta)^{-1/2}u,
\end{align*}
with 
\begin{align*}
\Sigma^N([\vartheta, \vartheta'];u) = N^{-1}\sum_{i = 1}^N \sum_{j = 1}^d \int_0^T &\nabla_\vartheta(c^{-1/2}b)^j([\vartheta,\vartheta'];t,X_t^i,\mu_t^{(N)})(\nabla_\vartheta(c^{-1/2}b)^j([\vartheta,\vartheta'];t,X_t^i,\mu_t^{(N)}))^\top dt
\end{align*}
which converges by Lemma \ref{lem: conv prob}
 to
 \begin{align*}\Sigma([\vartheta, \vartheta'];u) =  \sum_{j = 1}^d \int_0^T \int_{\R^d}& \nabla_\vartheta(c^{-1/2}b)^j([\vartheta, \vartheta'];t,x,\mu_t^\vartheta)(\nabla_\vartheta(c^{-1/2}b)^j([\vartheta, \vartheta'];t,x,\mu_t^{\vartheta}))^\top \mu_t(dx)dt
\end{align*}
in $\PP_\vartheta^N$-probability. Abbreviating further $\widetilde \Sigma^N(\vartheta;u) = \mathbb I_{\mathcal G}(\vartheta)^{-1/2}\Sigma^N([\vartheta, \vartheta'];u)\mathbb I_{\mathcal G}(\vartheta)^{-1/2}$ and $\widetilde \Sigma(\vartheta;u) = \mathbb I_{\mathcal G}(\vartheta)^{-1/2}\Sigma([\vartheta, \vartheta'];u)\mathbb I_{\mathcal G}(\vartheta)^{-1/2}$, we have
\begin{align}
\E_{\PP_\vartheta^N}\big[\mathcal Z_N(\vartheta;u)^{1/2}\big] & \leq \E_{\PP_\vartheta^N}\big[\exp\big(-\tfrac{3}{16} u^\top \widetilde\Sigma^N(\vartheta;u)u\big) \big]^{1/3} \nonumber\\
& \leq \PP_\vartheta^N\big(u^\top \widetilde\Sigma^N(\vartheta;u) u\leq \tfrac{1}{2} u^\top \widetilde\Sigma(\vartheta;u) u\big)^{1/3} + \exp\big(-\tfrac{1}{32}u^\top \widetilde\Sigma(\vartheta;u)u\big). \label{eq: cheby}
\end{align} 
The non-degeneracy assumption (recall Definition \ref{def: non deg}) ensures $u^\top \widetilde\Sigma(\vartheta;u)u \geq \mathfrak c|u|^2$ for some $\mathfrak c>0$ that does not depend on $\vartheta$ nor $u$ (but that depends on $\Theta$), hence the remainder term decays faster than any power of $|u|$. The first term in the right-hand side of \eqref{eq: cheby} is bounded above by
\begin{align*}
&\PP_\vartheta^N\big(\big|u^\top \big(\widetilde\Sigma^N(\vartheta;u)-\widetilde\Sigma(\vartheta;u)\big)u\big| \geq \tfrac{1}{2} u^\top \widetilde\Sigma(\vartheta;u) u\big)^{\tfrac{1}{3}} \\
& \leq \PP_\vartheta^N\Big(\big|u^\top \big(\widetilde\Sigma^N(\vartheta;u)-\widetilde\Sigma(\vartheta;u)\big)u\big| \geq \tfrac{1}{2}\mathfrak c|u|^2 \Big)^{\tfrac{1}{3}} \\
& \leq C|u|^{-\tfrac{2m}{3}} \E_{\PP_\vartheta^N}\big[\big|u^\top \big(\widetilde\Sigma^N(\vartheta;u)-\widetilde\Sigma(\vartheta;u)\big)u\big|^m\big]^{\tfrac{1}{3}}
\end{align*}
for every $m >0$ by Markov's inequality and where we used the non-degeneracy assumption again. For $1 \leq \ell, \ell' \leq p$, introduce
\begin{align*}
\phi_{\ell, \ell'}([\vartheta, \vartheta'];t,x,\nu)  & = (\partial_{\vartheta_\ell}(c^{-1/2}b)([\vartheta, \vartheta'];t,x,\nu))^\top \partial_{\vartheta_\ell}(c^{-1/2}b)([\vartheta, \vartheta'];t,x,\nu),
\end{align*}
so that $\big(\Sigma^N(\vartheta;u)-\Sigma(\vartheta;u)\big)_{\ell,\ell'}$ is simply
\begin{align*}
N^{-1}\sum_{i = 1}^N\int_0^T\phi_{\ell, \ell'}([\vartheta,\vartheta'];t,X_t^i, \mu_t^{(N)})dt-
\int_0^T\int_{\R^d}\phi_{\ell, \ell'}([\vartheta,\vartheta'];t,x, \mu_t^\vartheta)\mu_t(dx)dt.
\end{align*}
By Lemma \ref{lem: conv prob}, we derive
$$\E_{\PP_\vartheta^N}\big[\big|\big(\Sigma^N(\vartheta;u)-\Sigma(\vartheta;u)\big)_{\ell,\ell'}\big|^m\big] \leq CN^{-\delta m},$$
therefore
\begin{align*}
 \E_{\PP_\vartheta^N}\big[\big|u^\top \big(\widetilde\Sigma^N(\vartheta;u)-\widetilde\Sigma(\vartheta;u)\big)u\big|^m\big]^{1/3} \leq C|u|^{2m/3}N^{-\delta m/3},
\end{align*}
and finally
\begin{equation} \label{eq: last cheby}
\PP_\vartheta^N\big(\big|u^\top \big(\widetilde\Sigma^N(\vartheta;u)-\widetilde\Sigma(\vartheta;u)\big)u\big| \geq \tfrac{1}{2} u^\top \widetilde\Sigma(\vartheta;u) u\big)^{\tfrac{1}{3}} 
\leq CN^{-\delta m/3}
\end{equation}
Pick $r \geq 1$. Combining \eqref{eq: cheby} and \eqref{eq: last cheby}, we infer
\begin{align*}
|u|^r\E_{\PP_\vartheta^N}\big[\mathcal Z_N(\vartheta;u)^{1/2}\big] & \leq C|u|^rN^{-\delta m/3}+|u|^r\exp(-\tfrac{\mathfrak c}{32}|u|^2).
\end{align*}
For $u = (N\mathbb I_{\mathcal G}(\vartheta))^{1/2}(\vartheta'-\vartheta)$ with $\vartheta'\in \Theta$, we have $|u| \leq CN^{1/2}$. The first term in the right-hand side of the previous estimate is thus bounded as soon a $m \geq 3r/(2\delta)$. The second term is bounded uniformly in $|u|$. We thus have established \eqref{eq: moment bound likelihood}.\\

\noindent {\it Step 3.} We are now ready to apply Theorem III.1.1 of \cite{ibragimov2013statistical} and gather several properties of the maximum likelihood estimator. Note that the continuity of the likelihood function $\vartheta \mapsto \mathcal L^N(\vartheta; X^{(N)})$ and the compactness of $\Theta$ ensures that a solution $\widehat \vartheta_N^{\,\tt{mle}}$ to \eqref{eq: def emv} exists.\\

The uniform LAN condition given in the proof of Theorem \ref{thm: LAN prop} is the Condition N1 of Chapter III \cite{ibragimov2013statistical}. The non-degeneracy assumption (according to Definition \eqref{def: non deg}) is related to the uniform use of Condition N2 of \cite{ibragimov2013statistical}. Step 1 and Step 2 are respectively Condition N3 and N4 of \cite{ibragimov2013statistical}. We may thus apply Theorem III.1.1 of \cite{ibragimov2013statistical} and we readily obtain Statement (i) of Theorem \ref{thm: MLE prop}. Statement (ii) is a consequence of Corollary III.1.1 of \cite{ibragimov2013statistical} while Statement (iii) is a consequence of Theorem III.1.3 of \cite{ibragimov2013statistical}.\\

The proof of Theorem \ref{thm: MLE prop} is complete.

 \section{Remaining proofs} \label{sec: remaining proofs}

\subsection{Proof of Proposition \ref{prop: smoothness parameter McKean}} \label{app: proof of smooth}
Anticipating the proof of Lemma \ref{lem: couplings}, we prove a slightly stronger result, namely
\begin{equation} \label{eq: anticipate}
\big(\E_{\PP_\vartheta^N}\big[|\overline{X}_t^{i,\vartheta}-\overline{X}_{t}^{i,\vartheta'}|^{r}\big]\big)^{1/r} \leq C|\vartheta-\vartheta'|
\end{equation}
for $r \geq 1$. Indeed, Proposition \ref{prop: smoothness parameter McKean} is then a consequence of
$$
\mathcal W_1(\mu_t^{\vartheta}, \mu_t^{\vartheta'}) \leq \E_{\PP_\vartheta^N}\big[|\overline{X}_t^{i,\vartheta}-\overline{X}_{t}^{i,\vartheta'}|\big] \leq \big(\E_{\PP_\vartheta^N}\big[|\overline{X}_t^{i,\vartheta}-\overline{X}_{t}^{i,\vartheta'}|^{2r}\big]\big)^{1/2r}, 
$$
for any $r \geq 1$.
From $\overline{X}_0^{i,\vartheta}=\overline{X}_{0}^{i,\vartheta'}$, we have
$$\overline{X}_t^{i,\vartheta}-\overline{X}_t^{i,\vartheta'} = \int_0^t\big(b(\vartheta;s,\overline{X}_s^{i,\vartheta},\mu_s^\vartheta)-b(\vartheta';s,\overline{X}_s^{i,\vartheta'}, \mu_s^{\vartheta'})\big)ds+ \int_0^t\big(\sigma(s,\overline{X}_s^{i,\vartheta})-\sigma(s,\overline{X}_s^{i,\vartheta'})\big)dB_s^{i,N,\vartheta}.$$
Using Assumption \ref{reg sigma} and the Burckholder-Davis-Gundy inequality, we infer
\begin{align*}
\E_{\PP_{\vartheta}^N}\Big[\Big| \int_0^t\big(\sigma(s,\overline{X}_s^{i,\vartheta})-\sigma(s,\overline{X}_s^{i,\vartheta'})\big)dB_s^{i,N,\vartheta}\Big|^{2r}\Big] 
&\leq C\E_{\PP_{\vartheta}^N}\Big[\Big( \int_0^t |\overline{X}_s^{i,\vartheta}-\overline{X}_s^{i,\vartheta'}\big|^2ds\Big)^{r}\Big] \\
&\leq C\E_{\PP_{\vartheta}^N}\Big[\int_0^t |\overline{X}_s^{i,\vartheta}-\overline{X}_s^{i,\vartheta'}\big|^{2r}ds\Big] 
\end{align*}
since $r\geq 1$. Thanks to the smoothness properties of $b$ granted by Assumptions \ref{reg b} and \ref{hyp : reg sigma b} and incorporating the previous estimate, we obtain
\begin{align*}
&\E_{\PP_{\vartheta}^N}\big[|\overline{X}_t^{i,\vartheta}-\overline{X}_t^{i,\vartheta'}|^{2r}\big] \\
&\leq C\int_0^t \E_{\PP_\vartheta^N}\big[|\vartheta-\vartheta'|^{2r}(1+|\overline{X}_s^{i,\vartheta}|^{2r_1r}+\mathfrak m_{r_2}(\mu_s^\vartheta)^{2r})+|\overline{X}_s^{i,\vartheta}-\overline{X}_s^{i,\vartheta'}|^{2r}+\mathcal W_1(\mu_{s}^{\vartheta},\mu_s^{\vartheta'})^{2r}\big]ds \\
& \leq C\big(|\vartheta-\vartheta'|^{2r} + \int_0^t  \big(\E_{\PP_\vartheta^N}\big[|\overline{X}_s^{i,\vartheta}-\overline{X}_s^{i,\vartheta'}|^{2r}+\mathcal W_1(\mu_{s}^{\vartheta},\mu_s^{\vartheta'})^{2r}\big]ds \big)
\end{align*}
where we used that $ \E_{\PP_\vartheta^N}[|\overline{X}_t^{i,\vartheta}|^{r'}] = \mathfrak m_{r'}(\mu_t^\vartheta)$ is bounded uniformly in $t \in [0,T]$ and $\vartheta \in \Theta$  for all values of $r' \geq 1$ by Lemma  \ref{lem : puissances limite}. We obtain
\eqref{eq: anticipate} for $2r$ by Gr\"onwall's lemma, hence for every $r \geq 1$ by Cauchy-Schwarz's inequality. The proposition follows.

\subsection{Proof of Proposition \ref{prop: indisting}} \label{proof: prop: indisting}

By Girsanov's theorem,
\begin{align*}
\E_{\overline{\PP}_{\vartheta}^{\otimes N}}\Big[\log\frac{d\overline{\PP}_\vartheta^{\otimes N}}{d\PP_\vartheta^N}\Big] 
& = \tfrac{1}{2} \E_{\overline{\PP}_{\vartheta}^{\otimes N}}\Big[\sum_{i = 1}^N \int_0^T |b(\vartheta; t,X_t^i,\mu_t^{(N)})-b(\vartheta; t,X_t^i, \mu_t^\vartheta)|^2dt \\
& =  \tfrac{1}{2}\sum_{i = 1}^N\int_0^T \E_{\overline{\PP}_{\vartheta}^{\otimes N}}\big[\big|N^{-1}\sum_{j = 1}^N\big(\widetilde b(\vartheta; t,X_t^i,X_{t}^j)-\E_{\overline{\PP}_\vartheta}[\widetilde b(\vartheta; t,\zeta, X_{t}^j)]_{\zeta = X_t^i}\big)\big|^2\big]dt.
\end{align*}
We plan to use the following decomposition
\begin{align*}
& N^{-1}\sum_{j = 1}^N\widetilde b(\vartheta; t,X_t^i,X_{t}^j)-\E_{\overline{\PP}_\vartheta}\big[\widetilde b(\vartheta; t,\zeta, X_{t}^j)\big]_{\zeta = X_t^i} \\
& = N^{-1}\big(\widetilde b(\vartheta; t,X_t^i,X_{t}^i)-\E_{\overline{\PP}_\vartheta}[\widetilde b(\vartheta; t,\zeta, X_{t}^i)]_{\zeta = X_t^i}\big)\\
&+\tfrac{N-1}{N} (N-1)^{-1}\sum_{j=1, j \neq i}^N\big(\widetilde b(\vartheta; t,X_t^i,X_{t}^j)-\E_{\overline{\PP}_\vartheta^{\otimes 2}}[\widetilde b(\vartheta; t,X_t^i, X_{t}^j)\,|X_{t}^i]\big). 
\end{align*}
Using the elementary inequality $(a+b)^2 \leq (1+\rho)a^2+(1+\rho^{-1})b^2$ valid for every $\rho >0$, 
we obtain
\begin{align*}
&\E_{\overline{\PP}_{\vartheta}^{\otimes N}}\big[\big|N^{-1}\sum_{j = 1}^N\widetilde b(\vartheta; t,X_t^i,X_{t}^j)-\E_{\overline{\PP}_\vartheta}[\widetilde b(\vartheta; t,\zeta, X_{t}^j)]_{\zeta = X_t^i}\big|^2\big]  \\
& \leq (1+\rho^{-1})N^{-2}\E_{\overline{\PP}_\vartheta}\big[|\widetilde b(\vartheta; t,X_t^1,X_{t}^1)|^2\big]+(1+\rho)\frac{N-1}{N^2}\E_{\overline{\PP}_\vartheta^{\otimes 2}}\big[|\widetilde b(\vartheta; t,X_t^1,X_{t}^2)|^2\big],
\end{align*}
and therefore
$$\limsup_{N \rightarrow \infty} \sup_{\vartheta \in \Theta}
\E_{\overline{\PP}_{\vartheta}^{\otimes N}}\Big[\log\frac{d\overline{\PP}_\vartheta^{\otimes N}}{d\PP_\vartheta^N}\Big] 
\leq \frac{1+\rho}{2}\sup_{\vartheta \in \Theta} \int_0^T \int_{\R^d \times \R^d} |\widetilde b(\vartheta; t,x,y)|^2(\mu_t^\vartheta \otimes \mu_t^\vartheta)(dx, dy)dt.$$
By assumption, 
$$\int_{\R^d\times \R^d} |\widetilde b(\vartheta; t,x,y)|^2(\mu_t^\vartheta \otimes \mu_t^\vartheta)(dx, dy) \leq C\big(1+\sup_{t \in [0,T], \vartheta \in \Theta}(\mathfrak m_{r_1}(\mu_t^\vartheta)+\mathfrak m_{r_2}(\mu_t^\vartheta))\big)$$
which is finite by Lemma \ref{lem : puissances limite}. We thus obtain \eqref{eq: entropy control}. In order to obtain \eqref{eq: pinsker}, we simply apply Pinsker's inequality:
\begin{align*}
\limsup_{N \rightarrow \infty}\sup_{\vartheta \in \Theta}\|\PP_{\vartheta}^N-\overline{\PP}_\vartheta^{\otimes N}\|_{TV}^2 & \leq \tfrac{1}{2} \limsup_{N \rightarrow \infty}\sup_{\vartheta \in \Theta}\E_{\overline{\PP}_{\vartheta}^{\otimes N}}\Big[\log\frac{d\overline{\PP}_\vartheta^{\otimes N}}{d\PP_\vartheta^N}\Big] \\
& \leq  \frac{1+\rho}{4}\sup_{\vartheta \in \Theta} \int_0^T \int_{\R^d \times \R^d} |\widetilde b(\vartheta; t,x,y)|^2(\mu_t^\vartheta \otimes \mu_t^\vartheta)(dx, dy).
\end{align*}

The conclusion follows by taking $\rho$ sufficiently small.
\subsection{Proof of Proposition \ref{prop: fisher cont}} \label{app: proof prop: fisher cont}
Let 
$$\big(\phi_{\ell, \ell'}(\vartheta ; t,x,\nu)\big)_{1 \leq \ell, \ell' \leq p} = \big(\partial_{\vartheta_\ell}(c^{-1/2}b)(\vartheta;t,x,\mu^{\vartheta}_t)^\top \partial_{\vartheta_{\ell'}}(c^{-1/2}b)(\vartheta;t,x,\mu^{\vartheta}_t)\big)_{1 \leq \ell, \ell' \leq p}.$$
We have
\begin{align*}
\big(\mathbb I_{\mathcal G}(\vartheta)\big)_{\ell, \ell'}-\big(\mathbb I_{\mathcal G}(\vartheta')\big)_{\ell, \ell'} & = \int_0^T \big(\E_{\overline{\PP}_\vartheta}\big[\phi_{\ell, \ell'}(\vartheta ; t,X_t, \mu_t^\vartheta)\big]-\E_{\overline{\PP}_{\vartheta'}}\big[\phi_{\ell, \ell'}(\vartheta' ; t,X_t, \mu_t^{\vartheta'})\big]\big)dt \\
& = \int_0^T \big(\E_{\PP_\vartheta^N}\big[\phi_{\ell, \ell'}(\vartheta ; t, \overline{X}_t^{1,\vartheta}, \mu_t^\vartheta)-\phi_{\ell, \ell'}(\vartheta' ; t,\overline{X}_t^{1,\vartheta'}, \mu_t^{\vartheta'})\big]\big)dt. 
\end{align*}
Thanks to the smoothness properties of $b$ and $\sigma$ granted by Assumptions \ref{reg sigma}, \ref{reg b} and \ref{hyp : reg sigma b}, we have
\begin{align*}
&\big|\phi_{\ell, \ell'}(\vartheta ; t, \overline{X}_t^{1,\vartheta}, \mu_t^\vartheta)-\phi_{\ell, \ell'}(\vartheta' ; t,\overline{X}_t^{1,\vartheta'}, \mu_t^{\vartheta'})\big| \\
& \leq C\big(|\vartheta-\vartheta'|+|\overline{X}_t^{1,\vartheta}-\overline{X}_t^{1,\vartheta'}|+\mathcal W_1(\mu_t^{\vartheta'}, \mu_t^\vartheta)\big)\big(1+|\overline{X}_t^{1,\vartheta}|^\alpha+|\overline{X}_t^{1,\vartheta'}|^{\alpha}+\mathfrak m_\beta(\mu_t^\vartheta)+\mathfrak m_\beta(\mu_t^{\vartheta'})\big).
\end{align*}
We have $\E_{\PP_\vartheta^N}[|\overline{X}_t^{1,\vartheta}|^r]   = \mathfrak m_r(\mu_t^\vartheta)$, which is uniformly bounded in $t\in [0,T], \vartheta \in\Theta$ for every $r \geq 1$ by Lemma \ref{lem : puissances limite}. Likewise $\E_{\PP_\vartheta^N}[|\overline{X}_t^{1,\vartheta'}|^{r}] \leq C( \mathfrak m_r(\mu_t^\vartheta)+\E_{\PP_\vartheta^N}[|\overline{X}_t^{1,\vartheta}-\overline{X}_t^{1,\vartheta'}|^r])$, therefore, by Cauchy-Schwarz's inequality
\begin{align*}
&\big|\E_{\PP_\vartheta^N}\big[\phi_{\ell, \ell'}(\vartheta,t, \overline{X}_t^{1,\vartheta}, \mu_t^\vartheta)-\phi_{\ell, \ell'}(\vartheta',t,\overline{X}_t^{1,\vartheta'}, \mu_t^{\vartheta'})\big]\big| \\
& \leq C\big(|\vartheta-\vartheta'|+(\E_{\PP_\vartheta^N}[|\overline{X}_t^{1,\vartheta}-\overline{X}_t^{1,\vartheta'}|^2])^{1/2} +\mathcal W_1(\mu_t^\vartheta, \mu_t^{\vartheta'})\big) (1+(\E_{\PP_\vartheta^N}[|\overline{X}_t^{1,\vartheta}-\overline{X}_t^{1,\vartheta'}|^2])^{1/2})
\end{align*}
and the Lipschitz smoothness follows by applying \eqref{eq: coupling 2} of Lemma \ref{lem: couplings} and Proposition \ref{prop: smoothness parameter McKean}. The convergence $N^{-1} \mathbb I_{\mathcal E^N}(\vartheta) \rightarrow \mathbb I_{\mathcal G}(\vartheta)$ is a simple consequence of Lemma \ref{lem: conv prob}.

\subsection{Proof of Proposition \ref{prop: eq-identif-nondeg}}  \label{proof prop: eq-identif-nondeg}
Note that for every $1 \leq \ell \leq p$:
$$\partial_{\vartheta_\ell}\ell^N(\vartheta; X^N)  = G^N(X^{(N)})_\ell + 2\sum_{\ell' = 1}^p\vartheta_{\ell'} H^N(X^{(N)})_{\ell, \ell'}.$$
Let $\vartheta, \vartheta'$ in $\Theta_0$ be such that $\PP_{\vartheta}^N=\PP_{\vartheta'}^N$ for every $N \geq 1$. (This also implies $\overline{\PP}_\vartheta = \overline{\PP}_{\vartheta'}$ by Lemma \ref{lem: conv prob} for instance.) Using the symmetry of $H^N(X^{(N)})$, we note that
$$\vartheta^\top H^N(X^{(N)})\vartheta - (\vartheta')^\top H^N(X^{(N)})\vartheta' = \sum_{\ell = 1}^p(\vartheta_\ell-\vartheta'_{\ell})\sum_{\ell'=1}^pH^N(X^{(N)})_{\ell, \ell'}(\vartheta_{\ell'}+\vartheta'_{\ell'}).$$
It follows that
\begin{align*}
0  = \ell^N(\vartheta; X^{(N)})-\ell^N(\vartheta';X^{(N)})  
& = \sum_{\ell=1}^p(\vartheta_\ell-\vartheta'_{\ell})\Big(G^N(X^{(N)})_\ell+2\sum_{\ell'=1}^pH^N(X^{(N)})_{\ell, \ell'}\tfrac{\vartheta_{\ell'}+\vartheta'_{\ell'}}{2}\Big)\\
& = \sum_{\ell = 1}^p  \big(G^N(X^{(N)})_\ell+2\sum_{\ell' = 1}^p \tfrac{\vartheta_{\ell'}+\vartheta'_{\ell'}}{2} H^N(X^{(N)})_{\ell, \ell'}\big)(\vartheta_\ell-\vartheta_\ell') \\
& = \sum_{\ell = 1}^p \partial_{\vartheta_\ell}\ell^N(\vartheta^\star;X^{(N)})\xi_\ell = \nabla_\vartheta \ell^N(\vartheta^\star; X^{(N)})^\top\xi,
\end{align*}
with $\xi = \vartheta -\vartheta'$ and $\vartheta^\star = \tfrac{1}{2}(\vartheta+\vartheta') \in\Theta_0$ by the convexity of $\Theta_0$ and that does not depend on $X^{(N)}$. Assume now that $\vartheta \neq \vartheta'$. This implies that for some $\xi \neq 0$, we have
$$0 = (\nabla_\vartheta \ell^N(\vartheta^\star ; X^{(N)})^\top\xi)^\top\nabla_\vartheta \ell^N(\vartheta^\star ; X^{(N)})^\top\xi = \xi^\top \mathbb I_{\mathcal E^N}(\vartheta^\star) \xi.$$
Thus $\mathbb I_{\mathcal E^N}(\vartheta^\star)$ is degenerate for every $N \geq 1$. Letting $N \rightarrow \infty$ and applying Proposition \ref{prop: fisher cont},  we infer that $\mathbb I_{\mathcal G}(\vartheta^\star)$ is degenerate as well, a contradiction. The conclusion follows for $\mathcal G$ likewise.

\section{Appendix} \label{sec: appendix}
\subsection{Proof of Lemma \ref{lem: moment bound}} \label{app: proof of moment bound}
By Assumption \ref{reg b} we have $b_0 =\sup_{t \in [0,T]} | b(\vartheta_0 ; t, 0, \delta_0)| < \infty$ for some $\vartheta_0 \in \Theta$. Combined with Assumption \ref{hyp : reg sigma b}, since $\Theta$ is compact,  we infer
$$|b(\vartheta ; s,X_s^i,\mu_s^{(N)})| \leq C(1+|X_s^i|+N^{-1}\sum_{i = 1}^N|X_s^i|)$$
uniformly in $s \in [0,T]$ and $\vartheta \in \Theta$. For $M>0$, define $\tau_M = \inf\{s\geq 0,\max_{1 \leq i \leq N}|X_s^i| \geq M\} \wedge T$ and note that $\tau_M$ is a $(\mathcal F_t)$-stopping time. We have
\begin{align*}
|X_{t \wedge \tau_M}^i| & \leq |X_0^i|+\int_0^{t \wedge \tau_M} |b(\vartheta, s,X_s^i,\mu_s^{(N)})|ds+\big|\int_0^{t \wedge \tau_M} \sigma(s,X_s^i)dB_s^{i,N,\vartheta}\big| \\
& \leq |X_0^i|+C\int_0^{t \wedge \tau_M}(1+|X_s^i|+N^{-1}\sum_{i = 1}^N|X_s^i|)ds+\big|\int_0^{t \wedge \tau_M} \sigma(s,X_s^i)dB_s^{i,N,\vartheta}\big| \\
&  \leq |X_0^i|+C\int_0^{t}(1+|X_{s\wedge \tau_M}^i|+N^{-1}\sum_{i = 1}^N|X_{s\wedge \tau_M}^i|)ds+\sup_{0 \leq t \leq T}\big|\int_0^{t} \sigma(s,X_s^i)dB_s^{i,N,\vartheta}\big|
\end{align*}
Taking $\PP_{\vartheta}^N$-expectation of order $r \geq 1$, we obtain
\begin{align*}
\E_{\PP_{\vartheta}^N}\big[|X_{t \wedge \tau_M}^i|^r\big] & \leq C\big(1+ \E_{\PP_{\vartheta}^N}\big[|X_0^i|^r\big]+\int_0^{t}(1+\E_{\PP_{\vartheta}^N}\big[|X_{s\wedge \tau_M}^i|^r\big]ds\big),
\end{align*}
using Jensen's inequality, the exchangeability of $\PP_{\vartheta}^N$ and the Burckholder-Davis-Gundy inequality together with Assumption \ref{reg sigma} to obtain
$$\E\big[\sup_{0 \leq t \leq T}\big|\int_0^{t} \sigma(s,X_s^i)dB_s^{i,N,\vartheta}\big|^r\big] \leq C\E\big[\big(\int_0^{t} |\sigma(s,X_s^i)|^2s\big)^{r/2}\big] \leq C'.$$
By Gr\"onwall's lemma, we infer
$$\E_{\PP_{\vartheta}^N}\big[|X_{t \wedge \tau_M}^i|^r\big] \leq (\E_{\PP_{\vartheta}^N}\big[|X_0^i|^r\big]+C')\exp(Ct).$$
Letting $M \rightarrow \infty$, we conclude by Fatou's lemma.

\subsection{Proof of Lemma \ref{lem : puissances limite}} \label{app: proof of moment limit}
 It is a slight variation of the proof of Lemma \ref{lem: moment bound}. By Theorem 4.21 in \cite{carmona2018probabilistic}, we have $\E_{\overline{\PP}_\vartheta}[\sup_{0 \leq t \leq T}|X_t^i|^r] < \infty$ for every $r \geq 1$. (Their argument is developed for $r =2$ but the extension to any $r \geq 1$ is straightforward.) Therefore, only the uniformity in $\vartheta$ requires a proof.  From  
$$\E_{\PP_\vartheta^N}\big[|\overline{X}_t^{i,\vartheta}|^r\big] \leq C\big(\E_{\PP_\vartheta^N}\big[|X_t^i-\overline{X}_t^{i,\vartheta}|^r\big] +\E_{\PP_\vartheta^N}\big[|X_t^{i}|^r\big] \big),$$
and the fact that $\E_{\PP_\vartheta^N}\big[|\overline{X}_t^{i,\vartheta}|^r\big]  = \E_{\overline{\PP}_\vartheta}\big[|X_t^{i,\vartheta}|^r\big]$,  Lemma  \ref{lem : puissances limite} is now a simple consequence of  \eqref{eq: last coupling} in Lemma \ref{lem: couplings} together with Lemma \ref{lem: moment bound}.

\subsection{Proof of Lemma \ref{lem: couplings}} \label{app: lem couplings}
\subsubsection*{Proof of \eqref{eq: coupling 1}} The first inequality is obvious. Then, since $X_{0}^i= X_0^{i,\vartheta'}$, we have
$$X_t^i-X_t^{i,\vartheta'} = \int_0^t\big(b(\vartheta;s,X_s^i,\mu_s^{(N)})-b(\vartheta';s,X_s^{i,\vartheta'}, \mu_s^{(N), \vartheta'})\big)ds+ \int_0^t\big(\sigma(s,X_s^i)-\sigma(s,X_s^{i,\vartheta'})\big)dB_s^{i,N,\vartheta}.$$
Thanks to the smoothness properties of $b$ and $\sigma$ granted by Assumptions \ref{reg sigma}, \ref{reg b} and \ref{hyp : reg sigma b}, taking expectation to the power $r$ on both side and applying the Burckholder-Davis-Gundy inequality, we infer 
\begin{align*}
&\E_{\PP_{\vartheta}^N}\big[|X_t^i-X_t^{i,\vartheta'}|^r\big] \\
&\leq C\int_0^t \E_{\PP_\vartheta^N}\big[|\vartheta-\vartheta'|^r(1+|X_s^i|^{r_1}+\mathfrak m_{r_2}(\mu_s^{(N)}))^r+|X_s^{i}-X_s^{i,\vartheta'}|^r+\mathcal W_1(\mu_{s}^{(N), \vartheta'},\mu_s^{(N)})^r \big]ds \\
& \leq C\big(|\vartheta-\vartheta'|^r + \int_0^t  \big(\E_{\PP_\vartheta^N}\big[|X_s^{i}-X_s^{i,\vartheta'}|^r+\mathcal W_1(\mu_{s}^{(N), \vartheta'},\mu_s^{(N)})^r \big]\big)ds \big)
\end{align*}
where we used that $ \E_{\PP_\vartheta^N}\big[\mathfrak m_{r_2}(\mu_t^{(N)})^r\big] \leq \E_{\PP_\vartheta^N}[|X_t^{i}|^{r_2r}]$ which is bounded uniformly in $t \in [0,T]$ and $\vartheta \in \Theta$ by Lemma  \ref{lem: moment bound}. Also, using the first part of  \eqref{eq: coupling 1}, namely
$$ \E_{\PP_\vartheta^N}\big[\mathcal W_1(\mu_{s}^{(N), \vartheta'},\mu_s^{(N)})^r\big] \leq CN^{-1}\sum_{i = 1}^N \E_{\PP_\vartheta^N}\big[|X_s^{i}-X_s^{i,\vartheta'}|^r\big]$$
and taking averages over $i = 1,\ldots, N$ on both sides, we infer
\begin{align*}
&N^{-1}\sum_{i = 1}^N\E_{\PP_{\vartheta}^N}\big[|X_t^i-X_t^{i,\vartheta'}|^r\big] \leq  C\big(|\vartheta-\vartheta'|^r  +\int_0^t N^{-1}\sum_{i =1}^N\E_{\PP_{\vartheta}^N}\big[|X_s^i-X_s^{i,\vartheta'}|^r\big]ds \big).
\end{align*}
We obtain the second part of \eqref{eq: coupling 1} by Gr\"onwall's lemma.
\subsubsection*{Proof of \eqref{eq: coupling 2}} The first inequality is obvious. The second part is simply \eqref{eq: anticipate} from the proof of Proposition \ref{prop: smoothness parameter McKean}.

\subsubsection*{Proof of \eqref{eq: first last coupling} and \eqref{eq: last coupling}} By triangle inequality,
\begin{equation} \label{eq: triangle wasser}
\mathcal W_1(\mu_t^{(N)}, \mu_t^\vartheta)  \leq \mathcal W_1(\mu_t^{(N)}, \overline{\mu}_t^{(N),\vartheta})+\mathcal W_1(\overline{\mu}_t^{(N),\vartheta}, \mu_t^{\vartheta}) 
 \leq N^{-1}\sum_{i = 1}^N|X_{t}^i-\overline{X}_t^{i,\vartheta}|+\mathcal W_1(\overline{\mu}_t^{(N),\vartheta}, \mu_t^{\vartheta}). 
\end{equation}
By Theorem 2 of \cite{FournierGuillin}, we have $\sup_{t \in [0,T], \vartheta \in \Theta}\E_{\PP_{\vartheta}^N}\big[\mathcal W_1(\overline{\mu}_t^{(N),\vartheta}, \mu_t^{\vartheta})^r\big] \leq CN^{-\delta r}$ for every $r \geq 1$ and some $\delta >0$. The value of $\delta$ depends on the dimension $d$ of the state space. The uniformity in $(t,\vartheta)$ follows in particular from the uniform moment bounds of Lemma \ref{lem : puissances limite} (see the conditions of Theorem 2 of \cite{FournierGuillin}). Therefore \eqref{eq: first last coupling} is a consequence of \eqref{eq: last coupling}. 

In order to establish \eqref{eq: last coupling}, since $X_0^i=\overline{X}_0^{i,\vartheta}$, we write
$$X_t^i-\overline{X}_t^{i,\vartheta} = \int_0^t\big(b(\vartheta;X_s^i,\mu_s^{(N)})-b(\vartheta;s,\overline{X}_s^{i,\vartheta}, \mu_s^{ \vartheta})\big)ds+ \int_0^t\big(\sigma(s,X_s^i)-\sigma(s,\overline{X}_s^{i,\vartheta})\big)dB_s^{i,N,\vartheta}.$$
Taking expectation to the power $r$ on both side and applying the Burckholder-Davis-Gundy inequality, we infer
\begin{align*}
\E_{\PP_{\vartheta}^N}\big[|X_t^i-\overline{X}_t^{i,\vartheta}|^r\big] & \leq C\int_0^t\E_{\PP_\vartheta^N}\big[|X_t^i-\overline{X}_s^{i,\vartheta}|^r+\mathcal W_1(\mu_s^{(N)}, \mu_s^{\vartheta})^r\big]ds \\
&  \leq C\int_0^t\E_{\PP_\vartheta^N}\big[|X_t^i-\overline{X}_s^{i,\vartheta}|^r+\mathcal W_1(\overline{\mu}_s^{(N),\vartheta}, \mu_s^{\vartheta})^r\big]ds \\
 & \leq \varepsilon_N+C\int_0^t\E_{\PP_\vartheta^N}\big[|X_t^i-\overline{X}_s^{i,\vartheta}|^r\big]ds,
\end{align*}
arguing as in \eqref{eq: triangle wasser}, and where $\varepsilon_N = CT\sup_{t \in [0,T], \vartheta \in \Theta} \E_{\PP_\vartheta^N}\big[\mathcal W_1(\overline{\mu}_s^{(N),\vartheta}, \mu_s^{\vartheta})^r\big] \leq CN^{-\delta r}$. Note that the constants in each line are uniformly bounded in $\vartheta \in \Theta$. We obtain \eqref{eq: last coupling} by Gr\"onwall's lemma.
\subsection{Proof of Lemma \ref{lem: conv prob}} \label{app: LGN general}
We plan to use the following decomposition
$$N^{-1}\sum_{i = 1}^N\int_0^t \phi(\vartheta_N; s, X_s^i, \mu_s^{(N)})ds - \int_0^t \int_{\R^d}\phi(\vartheta; s,x,\mu_s^\vartheta)\mu_s^\vartheta(dx)ds  = I+II+III,$$
with
\begin{align*}
I &= N^{-1}\sum_{i = 1}^N\int_0^t \big(\phi(\vartheta_N; s, X_s^i, \mu_s^{(N)})- \phi(\vartheta_N; s, \overline{X}_s^{i,\vartheta_N}, \mu_s^{\vartheta_N})\big)ds\\
II&= N^{-1}\sum_{i = 1}^N\int_0^t \big(\phi(\vartheta_N; s, \overline{X}_s^{i,\vartheta_N}, \mu_s^{\vartheta_N})-\phi(\vartheta; s, \overline{X}_s^{i,\vartheta}, \mu_s^{\vartheta})\big)ds\\
III& = N^{-1}\sum_{i = 1}^N \int_0^t \phi(\vartheta; s, \overline{X}_s^{i,\vartheta}, \mu_s^{\vartheta})ds - \E_{\overline{\PP}_\vartheta}\big[\int_0^t\phi(\vartheta; s, \overline{X}_s^{i,\vartheta}, \mu_s^{\vartheta})ds\big].
\end{align*}
Thanks to the properties of $\phi$ the term $I$ is bounded by a constant times 
\begin{align*}
&N^{-1}\sum_{i = 1}^N \int_0^t \big(|X_s^i-\overline{X}_s^{i,\vartheta_N}|+\mathcal W_1(\mu_s^{(N)}, \mu_s^{\vartheta_N})\big)\big(1+|\overline{X}_s^{i,\vartheta_N}|^\alpha+|X_s^i|^\alpha+\mathfrak{m}_\beta(\mu_s^{\vartheta_N})+\mathfrak{m}_\beta(\mu_s^{(N)})\big)ds\\
& \leq \Big( N^{-1}\sum_{i = 1}^N \int_0^t \big(|X_s^i-\overline{X}_s^{i,\vartheta_N}|+\mathcal W_1(\mu_s^{(N)}, \mu_s^{\vartheta_N})\big)^2ds\Big)^{1/2}\\
& \;\;\;\times C\Big(N^{-1}\sum_{i = 1}^N\int_0^t \big(1+|\overline{X}_s^{i,\vartheta_N}|^{2\alpha}+|X_s^i|^{2\alpha}+\mathfrak{m}_{2\beta}(\mu_s^{\vartheta_N})+\mathfrak{m}_{2\beta}(\mu_s^{(N)})\big)ds\Big)^{1/2}. 
\end{align*}
by Cauchy-Schwarz's inequality. Applying Cauchy-Schwarz's inequality again together with Jensen's inequality, the  $\PP_{\vartheta_N}^N$-expectation to the power $m$ of the first term is then bounded by a constant times
 \begin{align*}
&\E_{\PP_{\vartheta_N}^N}\Big[\Big|N^{-1}\sum_{i = 1}^N\int_0^t\big(|X_s^i-\overline{X}_s^{i,\vartheta_N}|^2+\mathcal W_1(\mu_s^{(N)}, \mu_s^{\vartheta_N})^2\big)ds\Big|^m\Big]^{1/2}\\
&\;\;\times \Big(\int_0^t\big(1+\E_{\PP_{\vartheta_N}^N}\big[|\overline{X}_s^{i,\vartheta_N}|^{2m\alpha}+|X_s^i|^{2m\alpha}\big]+\mathfrak{m}_{2m\beta}(\mu_s^{\vartheta_N})+\mathfrak{m}_{2m\beta}(\mu_s^{(N)})\big)ds\Big)^{1/2}.
\end{align*}
The first term is bounded by a constant times $N^{-m\delta}$ by \eqref{eq: first last coupling} and \eqref{eq: last coupling} of Lemma \ref{lem: couplings}.
Also the  $\PP_{\vartheta_N}^N$-expectation of $|\overline{X}_s^{i,\vartheta_N}|^{2m\alpha}$, $|X_s^i|^{2m\alpha}$ and $\mathfrak{m}_{2m\beta}(\mu_s^{(N)})$ is uniformly bounded in $s \in [0,T]$ by Lemma \ref{lem: moment bound} and so is $\mathfrak{m}_{2m\beta}(\mu_s^{\vartheta_N})$ by Lemma \ref{lem : puissances limite}. We conclude
$$\E_{\PP_{\vartheta_N}^N}\big[\big| I\big|^m\big]  \leq CN^{-\delta m}.$$

The second term $II$ is bounded by a constant times
$$ N^{-1}\sum_{i = 1}^N\int_0^t \big( |\vartheta_N-\vartheta| +|\overline{X}_s^{i,\vartheta_N}-\overline{X}_s^{i,\vartheta}|+\mathcal W_1(\mu_s^{\vartheta_N},\mu_s^{\vartheta})\big)ds.$$
Taking  $\PP_{\vartheta_N}^N$-expectation to the power $m$ and applying successively the first and second part of
 \eqref{eq: coupling 2} in Lemma \ref{lem: couplings}, we obtain 
 $$\E_{\PP_{\vartheta_N}^N}\big[\big| II\big|^m\big]  \leq C|\vartheta-\vartheta_N|^m.$$
Finally, the third and last term converges to $0$ by the law of large numbers, applying for instance Rosenthal's inequality for a precise bound in $N$. We obtain
$$\E_{\PP_{\vartheta_N}^N}\big[\big| III \big|^m\big]  \leq CN^{-m/2}.$$
The proof of Lemma \ref{lem: conv prob} is complete.

\bibliographystyle{plain}       
\bibliography{DMH3.bib}

\end{document}